\documentclass[a4paper,leqno,12pt]{article}

\usepackage[utf8]{inputenc}
\usepackage[T1]{fontenc}
\usepackage[english]{babel}

\usepackage{crimson}

\usepackage{geometry}

\usepackage{fancyhdr}

\usepackage{hyperref}
\hypersetup{
  colorlinks   = true, 
  urlcolor     = blue, 
  linkcolor    = blue, 
  citecolor   = black 
}

\usepackage[inline]{enumitem}

\usepackage{graphicx}

\usepackage{color}

\usepackage[intlimits]{amsmath}
\usepackage{amssymb, amsfonts}

\usepackage[thmmarks, amsmath, amsthm]{ntheorem}

\usepackage{MnSymbol}
\usepackage{mathbbol}

\usepackage{bbm}

\usepackage{mathrsfs}

\usepackage[all]{xy}

\numberwithin{equation}{section}
\newtheorem{theoremcounter}{theoremcounter}[section]

\theoremstyle{plain}

\newtheorem{corollary}[theoremcounter]{Corollary}
\newtheorem{lemma}[theoremcounter]{Lemma}
\newtheorem{proposition}[theoremcounter]{Proposition}

\newtheorem{theorem}[theoremcounter]{Theorem}

\theoremnumbering{Alph}
\newtheorem{introtheorem}{Theorem}

\theoremstyle{definition}

\newtheorem{definition}[theoremcounter]{Definition}

\theoremstyle{remark}

\newtheorem{notation}[theoremcounter]{Notation}

\newtheorem{remark}[theoremcounter]{Remark}

\usepackage{xargs}                      
\usepackage[pdftex,dvipsnames]{xcolor}  
\usepackage[colorinlistoftodos,prependcaption,textsize=tiny]{todonotes}
\newcommandx{\unsure}[2][1=]{\todo[linecolor=red,backgroundcolor=red!25,bordercolor=red,#1]{#2}}
\newcommandx{\change}[2][1=]{\todo[linecolor=blue,backgroundcolor=blue!25,bordercolor=blue,#1]{#2}}
\newcommandx{\info}[2][1=]{\todo[linecolor=OliveGreen,backgroundcolor=OliveGreen!25,bordercolor=OliveGreen,#1]{#2}}
\newcommandx{\improvement}[2][1=]{\todo[linecolor=Plum,backgroundcolor=Plum!25,bordercolor=Plum,#1]{#2}}

\newcommand{\cA}{\ensuremath{\mathcal{A}}}
\newcommand{\cB}{\ensuremath{\mathcal{B}}}

\newcommand{\cH}{\ensuremath{\mathcal{H}}}

\newcommand{\cK}{\ensuremath{\mathcal{K}}}
\newcommand{\cL}{\ensuremath{\mathcal{L}}}
\newcommand{\cM}{\ensuremath{\mathcal{M}}}

\newcommand{\cO}{\ensuremath{\mathcal{O}}}

\newcommand{\cS}{\ensuremath{\mathcal{S}}}

\newcommand{\cU}{\ensuremath{\mathcal{U}}}

\newcommand{\cZ}{\ensuremath{\mathcal{Z}}}

\newcommand{\rB}{\ensuremath{\mathrm{B}}}

\newcommand{\rD}{\ensuremath{\mathrm{D}}}

\newcommand{\rK}{\ensuremath{\mathrm{K}}}

\newcommand{\rM}{\ensuremath{\mathrm{M}}}

\newcommand{\rS}{\ensuremath{\mathrm{S}}}

\newcommand{\rV}{\ensuremath{\mathrm{V}}}

\newcommand{\rmc}{\ensuremath{\mathrm{c}}}
\newcommand{\rmd}{\ensuremath{\mathrm{d}}}

\newcommand{\rmr}{\ensuremath{\mathrm{r}}}

\newcommand{\vphi}{\ensuremath{\varphi}}

\newcommand{\eqstop}{\ensuremath{\, \text{.}}}
\newcommand{\eqcomma}{\ensuremath{\, \text{,}}}

\newcommand{\NN}{\ensuremath{\mathbb{N}}}
\newcommand{\ZZ}{\ensuremath{\mathbb{Z}}}

\newcommand{\RR}{\ensuremath{\mathbb{R}}}
\newcommand{\CC}{\ensuremath{\mathbb{C}}}

\renewcommand{\Re}{\ensuremath{\mathop{\mathrm{Re}}}}

\newcommand{\id}{\ensuremath{\mathrm{id}}}
\newcommand{\ra}{\ensuremath{\rightarrow}}
\newcommand{\lra}{\ensuremath{\longrightarrow}}
\newcommand{\hra}{\ensuremath{\hookrightarrow}}

\newcommand{\ot}{\ensuremath{\otimes}}

\newcommand{\Cstar}{\ensuremath{\mathrm{C}^*}}

\newcommand{\bo}{\ensuremath{\mathcal{B}}}

\newcommand{\supp}{\ensuremath{\mathop{\mathrm{supp}}}}
\newcommand{\im}{\ensuremath{\mathop{\mathrm{im}}}}

\newcommand{\Cstarred}{\ensuremath{\Cstar_\mathrm{red}}}

\newcommand{\Ltwo}{\ensuremath{{\offinterlineskip \mathrm{L} \hskip -0.3ex ^2}}}

\newcommand{\ltwo}{\ensuremath{\ell^2}}
\newcommand{\cc}{\ensuremath{\mathrm{c}_\mathrm{c}}}

\newcommand{\linfty}{\ensuremath{{\offinterlineskip \ell \hskip 0ex ^\infty}}}

\newcommand{\lspan}{\ensuremath{\mathop{\mathrm{span}}}}
\newcommand{\cspan}{\ensuremath{\mathop{\overline{\mathrm{span}}}}}

\newcommand{\authors}{Piotr Nowak $\bullet$ Sanaz Pooya $\bullet$ Sven Raum $\bullet$ Adam Skalski}
\renewcommand{\title}{Delocalised traces on Hecke algebras of right-angled, hyperbolic type }

\begin{document}


\thispagestyle{empty}

\noindent
\begin{minipage}{\linewidth}
  \begin{center}
    \textbf{\Large \title} \\
    \authors    
  \end{center}
\end{minipage}

\renewcommand{\thefootnote}{}
 \footnotetext{last modified on \today}

\vspace{2em}
\noindent
\begin{minipage}{\linewidth}
  \textbf{Abstract}.
  For every Hecke C*-algebra of right-angled, hyperbolic type, we construct a smooth subalgebra to which traces associated with arbitrary conjugacy classes in the associated Coxeter group extend.  We calculate the pairing with K-theory of the delocalised traces arising this way, showing that it is faithful on K-theory.
\end{minipage}


\section{Introduction}
\label{sec:introduction}

The study of delocalised traces forms an important tool in noncommutative geometry and index theory, particularly in the analysis of delocalised $\Ltwo$-invariants and higher indices associated with manifolds, discrete groups and their operator algebras. These traces, which extend the familiar notion of the canonical trace on group C*-algebras, allow one to distinguish K-theory classes that are indistinguishable by the standard trace. A foundational instance of their use appears in the work of Lott \cite{lott1999}, who investigated for example delocalised $\Ltwo$-Betti numbers and delocalised $\Ltwo$-torsion of hyperbolic manifolds.  The theoretical underpinning for the pairing of delocalised traces with K-theory classes in the context of hyperbolic groups was developed by Puschnigg \cite{puschnigg2010}, who constructed dense subalgebras of group C*-algebras which are at the same time stable under holomorphic functional calculus and host bounded traces that sum coefficients over possibly infinite conjugacy classes. This might be viewed as a natural continuation of the work of Jolissaint  \cite{jolissaint1989-k-theory}, who, inspired by the earlier results of Connes, produced natural smooth subalgebras contained in reduced group C*-algebras, remembering the K-theory of their C*-counterparts.  More generally, Puschnigg was able to prove that cyclic homology of the group algebra injects into local cyclic homology of the associated reduced group C*-algebra.

While this story is well-developed for group C*-algebras of hyperbolic groups, a parallel and increasingly important direction involves Hecke C*-algebras, which arise as deformations of group algebras associated with Coxeter systems. Historically, spherical and affine Hecke algebras appeared naturally in number theory and representation theory and as such have become a central part of a deep and well-developed theory.  They are closely connected to the representation theory of finite groups of Lie type and reductive p-adic groups, respectively.  Delorme-Opdam's work in \cite{delormeopdam2008} considered Hecke C*-algebra of affine types and introduced Schwartz algebras in this context.  Notably, by Bernstein's theorem, affine Hecke algebras are finitely generated modules over their centre, so that classical methods related the Fourier transform could be employed.  These Hecke-Schwartz algebras of affine type were instrumental in Opdam-Solleveld's approach to determine discrete series representations \cite{opdamsolleveld2010-discrete-series} and generalised principal series \cite{delormeopdam2011} of affine Hecke algebras.

Hecke operator algebras of indefinite type drew attention only after the introduction of Hecke von Neumann algebras in the context of weighted $\Ltwo$-cohomology of buildings by Davis-Dymara-Januszkiewicz-Okun \cite{davisdymarajanusykiewiczokun07} around two decades ago.  Specifically right-angled Hecke operator algebras since then gained prominence in the operator algebra community due to their rich structural features, which are at the same time amenable to a variety of established methods from the field.  Known features of right-angled Hecke C*-algebras include the rapid decay property \cite{caspersklisselarsen2021}, simplicity results \cite{caspersklisselarsen2021,klisse2023-simplicity} and K-theory calculations \cite{raumskalski2022}. 
 Moreover, Hecke operator algebras encode significant aspects of unitary representation theory, particularly through their role in classifying Iwahori-spherical unitary representations of groups acting strongly transitively on buildings such as certain completed Kac-Moody groups.  See e.g.\ \cite[Corollary B]{raumskalski2023}.

The cyclic homology of affine Hecke algebras, was clarified in work Baum-Nistor \cite{baumnistor2002} in the purely algebraic setup and by Solleveld \cite{solleveld2009-thesis,solleveld2009-hp} in the operator algebraic setup of Hecke-C*-algebras as well as in the setup of non-commutative geometry provided by Schwartz algebras.

In this work we consider Hecke algebras of right-angled, hyperbolic type.  In this setup, delocalised traces account for all of cyclic cohomology, so that a natural first step is to determine traces on such Hecke algebras.  This is achieved in the next theorem, which is also proven in the greater generality of generic Hecke algebras in Theorem~\ref{thm:delocalised-trace-identification}.  For simplicity we state it in a simplified form for Iwahori-Hecke algebras.
\begin{introtheorem}
  \label{thmintro:delocalised-trace-identification-simplified}
  Let $\cH$ be an Iwahori Hecke algebra of right-angled type $(W,S)$.
  \begin{enumerate}
  \item For every conjugacy class $\cO \subseteq W$, there is a unique trace $\vphi_\cO$ on $\cH$ that satisfies $\vphi_\cO(T_w) = \mathbb{1}_\cO(w)$ for every element $w \in W$ that is of minimal length in its conjugacy class.
  \item The assignment $(c_\cO)_{\cO} \mapsto \sum_\cO c_\cO \vphi_\cO$ defines an isomorphism between the space of complex sequences on the set of conjugacy classes of $W$ and the space of tracial functionals on $\cH$.
  \end{enumerate}
\end{introtheorem}
We mention that in order to prove this result, we establish in Theorem~\ref{thm:right-angled-cocentre-basis} that the cocentre of Iwahori Hecke algebra of right-angled type has a natural basis indexed by representatives of conjugacy classes.  The problem of describing the cocentre of Hecke algebras is well-known in the context of character theory.  See e.g. \cite{geckpfeiffer2000-characters} for spherical Coxeter types and \cite{henie2014} for affine Coxeter types.  The recent work \cite{chen2025-centralizers} explicitly states folklore conjectures about the cocentre for indefinite Coxeter types.

The next theorem extends Puschnigg's construction of Schwartz algebras for hyperbolic groups to the realm of right-angled Hecke algebras.  We remark that Moussong's characterisation of hyperbolic Coxeter groups \cite[Theorem B]{moussong1988} implies that a right-angled Coxeter group is hyperbolic if and only if its commutation graph does not contain any induced squares.
\begin{introtheorem}
  \label{introthm:Schwartz-algebra-existence}
  Let $(W,S)$ be a right-angled, hyperbolic Coxeter system and $\cH$ the complex Iwahori-Hecke algebra of type $(W,S)$ with deformation parameter $(q_s)_{s \in S} \in \RR_{> 0}^S$.  There is a  subalgebra $\cS(W, S, q) \subseteq \Cstarred(W, S, q)$, which
  \begin{itemize}
  \item is dense and closed under holomorphic functional calculus,
  \item and such that for every conjugacy class $\cO \subseteq W$, the trace $\vphi_\cO$ continuously extends to $\cS(W, S, q)$.
  \end{itemize}
\end{introtheorem}

From the point of view of operator algebras, one of the main purposes of delocalised traces is to determine K-theory classes via their pairing.  The K-theory of right-angled Hecke C*-algebras was calculated in \cite{raumskalski2022}.  We briefly recall that for every clique $C$ in the commutation graph of a right-angled Coxeter system there is a projection $p_C$, such that the classes $([p_C])_C$ freely generate K-theory.  In Remark~\ref{rem:finite-order-element-cliques}, we observe that these cliques are in natural bijection with finite order elements of a right-angled Coxeter system.  We obtain the following result, which determines the pairing of K-theory and delocalised traces.
\begin{introtheorem}
  \label{thmintro:pairing}
  Let $(W,S)$ be a right-angled, hyperbolic Coxeter system and $(q_s)_{s \in S} \in \RR_{> 0}^S$ a deformation parameter.  Denote by $\mathrm{Cliq}(\Gamma)$ the set of cliques of the commutation graph $\Gamma$ of $(W, S)$, identified with the set of conjugacy classes of finite order elements in $W$.  For a clique $C \in \mathrm{Cliq}(\Gamma)$ let $p_C$ be the associated projection and let $\vphi_C$ be the delocalised trace associated with $C$.  Then the pairing between K-theory of $\Cstarred(W, S, q)$ and delocalised traces of finite order elements is determined by the formula
  \begin{gather*}
    \vphi_C(p_D) =
    \begin{cases}
      0 & C \nsubseteq D\eqcomma \\
      \prod_{s \in C} \frac{1}{q_s^{1/2} + q_s^{-1/2}} \prod_{s \in D \setminus C} \frac{1}{1 + q_s} & C \subseteq D \eqstop
    \end{cases}
  \end{gather*}
  In particular, the pairing $\rK_0(\Cstarred(W,S,q)) \times \RR^{\mathrm{Cliq(\Gamma)}} \lra \RR$ is non-degenerate.
\end{introtheorem}

As mentioned above, in the last section of \cite{puschnigg2010} Puschnigg used his main results to obtain certain decompositions of cyclic cohomology for his smooth algebras and the group C*-algebras of hyperbolic groups, close in spirit to the computation of cyclic homology of group rings due to Burghelea \cite{burghelea1985, burghelea2023-erratum-for-1985}. This line of research was later continued, using the Jolissaint's rapid decay subalgebras, for example in \cite{jiogleramsey14} and in \cite{engel20}. Our results in principle open the way to obtaining similar decompositions of cyclic cohomology in the context of Hecke algebras, such as the Fr{\'e}chet algebra $\cS(W, S, q)$ appearing in Theorem \ref{introthm:Schwartz-algebra-existence}. However, as we obtain a faithful pairing between delocalised traces and K-theory already, we leave this direction of research for future work.  It would chime with the recent progress in understanding noncommutative geometric properties of the Iwahori-Hecke algebras such as in the recent work \cite{klisseperovic25}.

This paper is organised as follows: after this introduction, in the preliminaries we describe background material and fix notation, focusing on Coxeter groups and Hecke algebras.  Then in the third section, we investigate the combinatorics of the multiplication in right-angled Hecke algebras, which plays a major role in this work.  In the fourth section, we adapt work of Blackadar-Cuntz and Puschnigg on differential seminorms and quasi-derivations to the needed generality.  In the fifth section, we use this framework in order to exhibit specific smooth subalgebras of the Hecke algebras under consideration, to which the natural quasi-derivation extends.  In the sixth section, we describe traces on right-angled Hecke algebras, and in that last section combine this with the previous results on smooth subalgebras, to obtain delocalised traces on Hecke-Schwartz algebras and calculate their pairing with K-theory.

\subsection*{Acknowledgements}

SP was partially supported by European Research Council (ERC) under the European Union’s Horizon 2020 research and innovation programme (grant agreement no. 677120-INDEX) and the Knut and Alice Wallenberg Foundation (grant number KAW 2020.0252).

SR was partially supported by the Swedish Research Council (grant number 2018-04243) and the German Research Foundation (DFG project no. 550184791).

PN was supported by the National Science Center Grant Maestro-13 UMO-2021/42/A/ST1/00306.

AS was partially supported by the National Science Center Grant OPUS-29 UMO-2025/57/B/ST1/00057.

The authors would like to thank IMPAN and its B{\c e}dlewo Conference Center as well as the University to Potsdam for their hospitality during respective visits.

\section{Preliminaries}
\label{sec:preliminaries}

As this article combines the theory of Coxeter groups and Hecke algebras with noncommutative geometry and operator algebras, we provide in this section references to necessary background material and recall key notions that are used in our work.

\subsection{Coxeter groups}
\label{sec:coxeter-groups}

We refer to the books of Humphreys \cite{humphreys90}, Garrett \cite{Garrettbook} and Abramenko-Brown \cite{abramenkobrown2008-buildings} for background material on Coxeter groups.

A marked group is a pair $(\Gamma, S)$ of a group with a finite generating subset $S \subseteq \Gamma$.  We write $S^*$ for the monoid of words with letters in $S$ and $|s_1 \dotsm s_n| = n$ for the length of a word in $S^*$.  Assuming $S$ is closed under inverses, for every $g \in \Gamma$ there is a word $s_1 \dotsm s_n \in S^*$ such that $s_1 \dotsm s_n = g$ in $\Gamma$.  We say that $s_1 \dotsm s_n$ represents $g$.  If $n$ is the minimal length of a word representing $g$, we say that $s_1 \dotsm s_n$ is reduced and write $|g| = n$.  The word $s_1 \dotsm s_n$ starts with $s_1$ and ends with $s_n$.  For $1 \leq k \leq n$, a word of the form $s_1 \dotsm s_k$  is called an initial piece and a word of the form $s_k \dotsm s_n$ is called a terminal piece of $s_1 \dotsm s_n$.

A Coxeter system is a marked group $(W, S)$ which admits a presentation
\begin{gather*}
  W = \langle S \mid (st)^{m_{st}} = e \rangle
\end{gather*}
for some matrix $(m_{st})_{s,t} \in \rM_S(\NN \cup \{\infty\})$ satisfying
\begin{itemize}
\item $m_{ss} = 1$ for all $s \in S$,
\item $m_{st} \geq 2$ for all $s, t \in S$ with $s \neq t$, and
\item $m_{st} = m_{ts}$ for all $s,t \in S$.
\end{itemize}
It is a fact that $m_{st} = \mathrm{ord}(st)$ for all $s,t \in W$, so that we can refer to these numbers without ambiguity. Note that $(st)^\infty=1$ is interpreted as the empty relation. 

A Coxeter system $(W, S)$  is called right-angled, if $m_{st} \in \{1,2,\infty\}$ for all $s,t \in S$.  The commutation graph of a right-angled Coxeter system $(W, S)$ is by definition the simplicial graph whose vertex set is $S$ and in which $s,t \in S$ are adjacent if and only if $m_{st} = 2$ holds, that is $st = ts$ in $W$.

For a subset $T \subseteq S$, we denote by $W_T \unlhd W$ the subgroup generated by $T$. A subgroup of the form $W_T$  is called a standard parabolic subgroup of $W$.

For every element $w \in W$ in a Coxeter system $(W, S)$, there is a subset $S(w) \subseteq S$ such that every reduced word $s_1 \dotsm s_n$ representing $w$ satisfies $S(w) = \{s_1, \dotsc, s_n\}$.  We say that the letters in $S(w)$ appear in $w$. A letter $s \in S$ is said to be an initial letter of the group element $w \in W$, if $|sw| < |w|$ holds.  Similarly, $s$ is a terminal letter of $w$ if $|ws| < |w|$. An expression $w = x_1 x_2 x_3 \dotsc x_k$ will be called a reduced decomposition of $w \in W$ if $x_1, \dotsc, x_k \in W \cup S^*$, the equality holds in $W$ and $|w| = |x_1| + \dotsm + |x_k|$ holds.  We will make use of the following fact for right-angled Coxeter systems, which is a consequence of Tits' solution to the word problem in Coxeter groups.
\begin{proposition}
  \label{prop:racs-cancellation}
  Let $(W, S)$ be a right-angled Coxeter system, let $s_1 \dotsm s_n$ be a reduced word in $S$ and let $s \in S$ be a terminal letter of the associated group element.  Then there is a unique index $1 \leq i \leq n$ such that $s = s_i$ and $[s_j, s] = e$ for all $i \leq j \leq n$. 
\end{proposition}

We will use the Bruhat order on elements of a Coxeter system, described for example in \cite[Sections 5.9 and 5.10]{humphreys90} or \cite[Section 1.8]{Garrettbook}.  Given a Coxeter system $(W, S)$ and $w,x \in W$, we write $x \leq w$ if there is a reduced expression $w = s_1\dotsm s_n$ and indices $1 \leq i_1 < \dotsm < i_m \leq n$ such that $x = s_{i_1} \dotsm s_{i_m}$.

\subsubsection*{Spherical subsets and their centralisers}

A subset $T \subseteq S$ for a Coxeter system $(W, S)$ is called spherical, if the standard parabolic subgroup $W_T$ is finite.

Let $(W, S)$ be a right-angled Coxeter system and $w \in W$ a product of pairwise commuting elements from $S$.  Then the centraliser of $w$ in $W$ can be described using Tits' solution to the word problem as described for example in \cite[Theorem 2.33]{abramenkobrown2008-buildings}.
\begin{gather*}
  \cZ_W(w) = \langle s \in S \mid \forall t \in S(w) \colon [s,t] = e \rangle
  \eqstop
\end{gather*}
In particular, $\cZ_W(w)$ is a standard parabolic subgroup.  For a subset $T \subseteq S$ we write $\mathrm{Comm(T)} \subseteq S$ for the set of generators commuting with all elements from $T$.  We abbreviate $\mathrm{Comm}(w) = \mathrm{Comm}(S(w))$.  Then $\cZ_W(w) = W_{\mathrm{Comm}(w)}$.

Let us fix the next piece of information as a remark for later reference.
\begin{remark}
  \label{rem:finite-order-element-conjugacy}
  By \cite[Proposition 2.87]{abramenkobrown2008-buildings}, every finite subgroup of a Coxeter group is conjugate to a subgroup of a spherical standard parabolic subgroup.  In particular, an element $w \in W$ has finite order if and only if it is conjugate into a spherical standard parabolic.
\end{remark}

Since  $\mathrm{ord}(st) = m_{st}$ for all $s,t \in S$, it follows that a standard parabolic subgroup $W_T \subseteq W$ in a right-angled Coxeter system $(W, S)$ is finite if and only if it is isomorphic with a direct sum $(\ZZ/2\ZZ)^{\oplus T}$.

The following remark will be used in Section~\ref{sec:traces-schwartz}.
\begin{remark}
  \label{rem:finite-order-element-cliques}
  Let $\Gamma$ be the commutation graph of a right-angled Coxeter system $(W, S)$.  Denote by $\mathrm{Cliq}(\Gamma)$ the set of cliques in $\Gamma$, that is the set of complete subgraphs of $\Gamma$ (including the empty graph).  Identifying vertices of $\Gamma$ with $S$, we consider the map $C \mapsto \prod_{s \in \rV(C)} s$.  By Remark~\ref{rem:finite-order-element-conjugacy} every finite order element is conjugate to an element in the image of this map.  Further, if $\prod_{s \in \rV(C)} s$ and $\prod_{s \in \rV(D)} s$ are conjugate, they must be equal by Tits solution to the word problem.  It follows that $C \mapsto \prod_{s \in \rV(C)} s$ defines a bijection between $\mathrm{Cliq}(\Gamma)$ and the set of conjugacy classes of finite order elements of $W$.
\end{remark}

\subsubsection*{Conjugacy classes}

We continue to consider a right-angled Coxeter system $(W, S)$ and summarise how one can describe conjugacy classes in $W$.   For simple reference in this article, we state the following theorem which summarises work on conjugacy classes in Coxeter groups from \cite{marquis2014-straight,marquis2020-structure-conjugacy-class,marquis2021-cyclically-reduced}.  Recall that a cyclic reduction of an element $w \in W$ is an element $w ' = s_k \dotsm s_1 w s_1 \dotsm s_k \in W$ such that $|s_{i+1} \dotsm s_1 w s_1 \dotsm s_{i+1}| < |s_i \dotsm s_1 w s_1 \dotsm s_i|$ for all $1 \leq i < k$.  We say $w$ cyclically reduced to $w'$.  If there is no cyclic reduction of $w$, we call it cyclically reduced.  Given a cyclically reduced element $w$ with reduced expression $w = s_1 \dotsc s_n$, every element $s_i \dotsm s_n s_1 \dotsm s_{i + 1} \in W$ for $1 \leq i \leq n$ is called a cyclic shift of $w$.
\begin{theorem}[Marquis]
  \label{thm:cyclic-reduction-shifts-right-angled}
  Let $(W, S)$ be a right-angled Coxeter system and $w \in W$.  Then the following statements hold true:
  \begin{itemize}
  \item every element $w$ can be cyclically reduced to a minimal length element in its conjugacy class;
  \item every pair of minimal length elements in the same conjugacy class are cyclic shifts of each other.
  \end{itemize}
\end{theorem}
\begin{proof}
  We first observe that every element in a right-angled Coxeter group satisfies Marquis' property $(\mathbf{Cent})$ by Tits solution to the word problem.  Now the first statement is a consequence of \cite[Corollary C]{marquis2014-straight}.  The second statement follows from a combination of \cite[Theorem A]{marquis2014-straight} treating the case of infinite order elements and the elementary Remark~\ref{rem:finite-order-element-cliques} treating the finite order case.
\end{proof}

\subsection{Hecke algebras}
\label{sec:hecke-algebras}

\begin{definition}
  \label{def:hecke-algebra}
  Let $(W, S)$ be a Coxeter system, let $R$ be a ring and let $a = (a_s)_{s \in S}$, $b = (b_s)_{s \in S}$ be elements in $R$ such that $(a_s, b_s) = (a_t, b_t)$ for all $s,t \in S$ that are conjugate in $W$.  Then the \emph{generic Hecke algebra} of type $(W, S)$ over $R$ with deformation parameters $a$ and $b$ is the unique $R$-algebra $\cH(W, S, a, b)$ which is free as a module over $R$ with designated basis $(T_w)_{w \in W}$ and multiplication determined by the following formula:
  \begin{gather*}
    T_s T_w =
    \begin{cases}
      T_{sw}  & \text{if } |sw| > |w| \\
      a_s T_w + b_s T_{sw} & \text{if } |sw| < |w|
    \end{cases}
    \qquad
    \text{ for  }
    s \in S, w \in W
    \eqstop
  \end{gather*}

  Given elements $q = (q_s)_{s \in S} \in \RR_{> 0}^S$ such that $q_s = q_t$ for all $s,t \in S$ that are conjugate in $W$, we put $p_s = q_s^{1/2} - q_s^{-1/2}$, $p = (p_s)_{s \in S}$ and let the \emph{Iwahori-Hecke algebra} with parameter $(q_s)_{s \in S}$ be defined as the complex generic Hecke algebra $\CC(W, S, q) = \cH(W, S, p, 1)$.

  Extending the action of an Iwahori-Hecke algebra by left multiplication on itself to the Hilbert space completion with respect to the inner product satisfying $\langle T_x, T_w \rangle = \delta_{x,w}$, we obtain the \emph{reduced Hecke C*-algebra} $\Cstarred(W, S, q)$ as the completion of $\CC(W, S, q)$ in the operator norm.
\end{definition}

\begin{remark}
  \label{rem:ra-hecke-algebras-parameters}
  In this article, we will only consider Hecke algebras of right-angled type $(W,S)$.  In this case, all tuples $(a_s)_{s \in S}, (b_s)_{s \in S}$ can be used to define a Hecke algebra, since different generators from $S$ are non-conjugate in $S$.  This is an immediate consequence of Tits' solution to the word problem in Coxeter systems.
\end{remark}

Generic Hecke algebras of right-angled type are often isomorphic to the group algebra of their Coxeter group.  For a reference in the case of Iwahori-Hecke algebras see e.g. \cite[Proposition 4.2]{caspersklisselarsen2021}.  As some of our results are phrased in the more general setup of generic Hecke algebras, we provide a proof in this case too.
\begin{lemma}
  \label{lem:order-preserving-base-change}
  Let $(I, \leq)$ be a partially ordered set, $R$ a domain and $(v_i)_{i \in I}$ an $R$-linearly independent family in an $R$-module $M$.  Assume that $(w_i)_{i \in I}$ is another family in $M$ such that $w_i = \sum_{j \leq i} c_{i,j} v_j$ for certain coefficients $c_{i,j} \in R$ satisfying $c_{i,i} \neq 0$.  Then $(w_i)_i$ is linearly independent.
\end{lemma}
\begin{proof}
  Consider a linear combination $\sum d_i w_i = 0$.  Then $\sum_{i,j \in I, j \leq i} d_i c_{i,j}v_j = 0$.  Assuming that $d_i \neq 0$ for some $i$, we let $i_0$ be a maximal element with this property.  If $c_{j,i_0} \neq 0$, we infer that $j \geq i_0$.  If additionally $d_j \neq 0$, then maximality implies $j = i_0$.   So the coefficient corresponding to $v_{i_0}$ in the expansion above equals $d_{i_0}c_{i_0,i_0}$. Thus   
  \begin{gather*}
    0 = (\sum_{i,j \in I, j \leq i} d_i c_{i,j}v_j)_{i_0} = d_{i_0}c_{i_0,i_0}
  \end{gather*}
  implying that $d_{i_0} = 0$, as $R$ is a domain and $c_{i_0,i_0} \neq 0$ by assumption.  This contradicts the choice of $i_0$ and completes the proof.
\end{proof}

\begin{proposition}
  \label{prop:isomorphism-ra-hecke-algebra}
  Let $(W,S)$ be a right-angled Coxeter system, let $R$ be a domain of characteristic different from $2$ and let $(a_s)_{s \in S}$, $(b_s)_{s \in S}$ be elements in $R$ such that $b_s + \frac{a_s^2}{4}$ is invertible and has a root in $R$ (which we shall fix).  Then there is an isomorphism of unital $R$-algebras $\vphi\colon R W \lra \cH(W, S, (a_s)_{s \in S}, (b_s)_{s \in S})$ satisfying $\vphi(u_s) = c_s T_s + d_s$ with
  \begin{gather*}
    c_s = (b_s + \frac{a_{s}^2}{4})^{-1/2}
    \quad
    \text{ and }
    d_s = -\frac{c_s a_s}{2}, \;\; s \in S
    \eqstop
  \end{gather*}
  It satisfies
  \begin{gather*}
    \vphi(u_w) = \sum_{w' \leq w} c_{w',w} T_{w'}
  \end{gather*}
  for certain coefficients $c_{w',w} \in R$ satisfying $c_{w,w} \in R^\times$.
\end{proposition}
\begin{proof}
  We write $\cH = \cH(W, S, (a_s)_{s \in S}, (b_s)_{s \in S})$. The elements $c_sT_s + d_s$ are invertible and self-inverse in $\cH$.  Indeed, we have
  \begin{gather*}
    (c_sT_s - d_s)^2
    =
    c_s^2 (a_s T_s + b_s) + 2c_sd_s T_s + d_s^2
    =
    c_s^2 (a_s T_s + b_s) - a_sc_s^2 T_s + c_s^2\frac{a_s^2}{4}
    =
    c_s^2(b_s + \frac{a_s^2}{4})
    =
    1
    \eqstop
  \end{gather*}
  Assuming $[s,t] = e$ in $W$, we also obtain $[c_sT_s + d_s, c_t T_t + d_t]  = 0$ in $\cH$.  So there is a uniquely determined $R$-algebra homomorphism $\vphi \colon R W \lra \cH$ satisfying $\vphi(u_s) = c_s T_s + d_s$ for all $s \in S$.  It is clear that $\vphi(u_w) = \sum_{w' \leq w} c_{w',w} T_{w'}$ for certain coefficients $c_{w',w} \in R$, with $c_{w,w} \neq 0$.  Since the $c_s, s \in S$ are invertible by assumption, the image of $\vphi$ is a unital $R$-subalgebra of $\cH$ containing its generators $T_s, s \in S$.  It follows that $\vphi$ is surjective.   Further, Lemma~\ref{lem:order-preserving-base-change} implies that the family $(\vphi(u_w))_{w \in W}$ is $R$-linearly independent, which implies that $\vphi$ is injective.
\end{proof}

Note that the isomorphism above need not extend to the natural reduced C*-completions (when these are available); this is the key feature making the study of $q$-deformed Hecke operator algebras interesting.

\subsection{Unconditional seminorms}
\label{sec:unconditional-seminorms}

Below we recall the notion of unconditional seminorms introduced by Bost and Lafforgue. In Puschnigg's work on smooth subalgebras of reduced group C*-algebras of hyperbolic groups \cite[Definition 2.1]{puschnigg2010}, they play a crucial role.  We also recall the notion of unconditional Fr{\'e}chet algebras from \cite[Definition 2.6]{puschnigg2010} needed for our work.
\begin{definition}
  \label{def:unconditional-seminorm}
  A seminorm $\| \cdot \|$ on a  complex vector space with a fixed basis, $(V, (v_i)_{i \in I})$, is unconditional if $\|\sum_i a_i v_i\| \leq \|\sum_i b_i v_i\|$ follows whenever $|a_i| \leq |b_i|$ holds for all $i \in I$.  A complex  algebra  with a fixed basis $(A, (a_i)_{i \in I})$ is called unconditional, if the kernel function of its multiplication has positive coefficients; in other words, if $a_i a_j = \sum_k c_{i,j}^k a_k$ for some $c_{i,j}^k \geq 0$.  The completion of an unconditional algebra with respect to a family of unconditional seminorms is called an unconditional Fr{\'e}chet algebra.
\end{definition}

  Given based vector spaces $(U, (u_i)_{i \in I})$ and $(V, (v_j)_{j \in J})$ with unconditional seminorms, \cite[Definition and Lemma 2.5]{puschnigg2010} introduces the largest unconditional seminorm on $(U \ot V, (u_i \ot v_j)_{(i,j) \in I \times J})$ satisfying $\|u \otimes v\|_{\mathrm{uc}} = \|u\| \|v\|$ for all $u \in U$, $v \in V$.  This leads to the notion of unconditional tensor product of unconditional Fr{\'e}chet spaces, which will be used in Section~\ref{sec:smooth-subalgebras-hecke} and onwards.  As shown in \cite[Lemma 2.7]{puschnigg2010}, the multiplication on an unconditional Fr{\'e}chet algebra is bounded with respect to the unconditional tensor product.

\section{Combinatorics of multiplication in right-angled Hecke algebras}
\label{sec:combinatorics-hecke-algebra}

In this section, we describe the multiplication of basis elements $(T_w)_{w \in W}$ in right-angled Hecke algebras in terms of word combinatorics.  We are interested in a combinatioral description of a small subset $A \subseteq W$ such that $T_wT_x$ is supported in $A$, that is $T_wT_x = \sum_{z \in A} c_zT_z$ for certain coefficients $c_z$.  While this question has a simple answer in group algebras, a description of multiplication in Hecke algebras has to take into account the fact that if $ws$ is shorter than $w$ then the product $T_wT_s$ admits the extra term proportional to $T_w$, more specifically $T_w T_s = a_s T_w + b_s T_{ws}$.  As this occurs precisely if the group element $w$ ends with the letter $s$, we think of skipping cancellation to obtain the extra term.  This perspective motivates the following terminology.
\begin{notation}
  \label{not:skipping-multiplication}
  Let $(W, S)$ be a Coxeter system.  Let $w = s_1 \dotsm s_m$ and $x = t_1 \dotsm t_n$ be reduced words in $W$.  Given $1 \leq j_1 < \dotsm < j_k \leq n$ such that for all $a \in \{1, \dotsc, k\}$ the group element $s_1 \dotsm s_m t_1 \dotsm \hat t_{j_1} \dotsm \hat t_{j_{a - 1}} \dotsm t_{j_a - 1}$  ends with $t_{j_a}$, we say that $s_1 \dotsm s_m t_1 \dotsm \hat t_{j_1} \dotsm \hat t_{j_k} \dotsm t_n$ arises from \emph{multiplying the word $s_1 \dotsm s_m$ with $t_1 \dotsm t_n$ while skipping cancellation at positions $j_1,\dotsc, j_k$}.  We then say that \emph{when multiplying $w$ and $x$, the cancellation can be skipped at positions $j_1,\dotsc, j_k$}.
\end{notation}
Let us make an additional comment on the precise formulation in the previous notation.  It picks up the definition of products $T_wT_t$ for a Coxeter generator $t$ and applies it iteratively to a product $T_wT_x = T_wT_{t_1}T_{t_2} \dotsm T_{t_n}$.  This leads to the description of elements that arise from multiplying $T_{s_1 \dotsm s_m t_1 \dotsm \hat t_{j_1} \dotsm \hat t_{j_{a - 1}} \dotsm t_{j_a - 1}}$ with the next factor of the term $T_{t_{j_a}} T_{t_{j_a + 1}} \dotsm T_{t_n}$.

\begin{theorem}
  \label{thm:skipping-cancellation-normal-form}
  Let $(W, S)$ be a right-angled Coxeter system and let $w = s_1 \dotsm s_m$, $x= t_1 \dotsm t_n$ be reduced words such that when multiplying them, cancellation can be skipped at positions $1\leq j_1 < \dotsm < j_k \leq n$.  Then
  \begin{itemize}
  \item $[t_{j_a}, t_{j_b}] = e$ for all $a,b \in \{1, \dotsc, k\}$;
  \item there exist words $w_1, x_1, x_2 \in W$ such that the decompositions $w = w_1 t_{j_1} \dotsm t_{j_k} x_1^{-1}$ and $x = x_1 t_{j_1} \dotsm t_{j_k} x_2$ are reduced and further  $w_1 t_{j_1} \dotsm t_{j_k} x_2$ is a reduced expression.
  \end{itemize}
\end{theorem}
\begin{proof}
  We prove the given statement by induction on $k \in \NN_0$. For $k = 0$ the statement follows from standard Coxeter combinatorics.  Assume it holds for some $k \in \NN_0$ and consider reduced words $w = s_1 \dotsm s_m$ and $x = t_1 \dotsm t_n$ such that when multiplying them, cancellation can be skipped at positions $j_1 < \dotsm < j_k < j_{k+1}$.  Then cancellation can be skipped at positions $j_1, \dotsc, j_k$ when multiplying the words $s_1 \dotsm s_m$ and $t_1 \dotsm t_{j_{k+1} - 1}$.  By induction hypothesis we infer that $[t_{j_a}, t_{j_b}] = e$ holds for all $a,b \in \{1, \dotsc, k\}$ and there are reduced decompositions $w = w_1 t_{j_1} \dotsm t_{j_k} x_1^{-1}$ and $x = x_1 t_{j_1} \dotsm t_{j_k} x_2 t_{j_{k+1}} x_3$ such that $w_1 t_{j_1} \dotsm t_{j_k} x_2$ is a reduced expression and $x_1 t_{j_1} \dotsm t_{j_k} x_2  = t_1 \dotsc t_{j_{k+1} - 1}$ as elements of $W$.  As cancellation can also be skipped at $j_{k+1}$ when multiplying the chosen words representing $w$ and $x$, the group element $w_1 t_{j_1} \dotsm t_{j_k} x_2$ ends with $t_{j_{k+1}}$.  As $x_1 t_{j_1} \dotsm t_{j_k} x_2 t_{j_{k+1}} x_3$ is reduced, also the subword $t_{j_1} \dotsm t_{j_k} x_2 t_{j_{k+1}}$ is reduced.  Hence, $t_{j_{k+1}}$ is a terminal letter of $w_1$ and by Proposition~\ref{prop:racs-cancellation} it commutes with all letters appearing in $t_{j_1} \dotsm t_{j_k} x_2$.  It follows that $[t_{j_a}, t_{j_b}] = e$ holds for all $a,b \in \{1, \dotsc, k+1\}$.  Further, removing a suitable occurrence of $t_{j+1}$ from $w_1$, we obtain a subexpression $w_1' \leq w_1$ such that $w = w_1' t_{j_1} \dotsm t_{j_k} t_{j_{k+1}} x_1^{-1}$ is a reduced decomposition and $w_1' t_{j_1} \dotsm t_{j_k} t_{j_{k+1}} x_2$ is a reduced expression.  Also notice that $x = x_1 t_{j_1} \dotsm t_{j_k} t_{j_{k+1}} x_2 x_3$ is a reduced decomposition.

  If $x_3$ is the empty word, the proof is finished.  Otherwise, let $x_3 = u_1 \dotsm u_l$ be a reduced word and let $i \in \{0, 1, \dotsc, l\}$ be maximal such that
  \begin{gather*}
    |w_1' t_{j_1} \dotsm t_{j_k} t_{j_{k+1}} x_2u_1 \dotsm u_i| = |w_1' t_{j_1} \dotsm t_{j_k} t_{j_{k+1}} x_2| + i
    \eqstop
  \end{gather*}
  If $i < l$, we infer from Proposition~\ref{prop:racs-cancellation} that $u_{i+1}$ commutes with the letters of $t_{j_1} \dotsm t_{j_{k+1}} x_2 u_1 \dotsc u_i$ and it is a terminal letter of $w_1'$.  So removing an occurrence of $u_{i+1}$ from $w_1'$, we obtain a subexpression $w_1'' \leq w_1'$  such that $w = w_1'' t_{j_1} \dotsm t_{j_k} t_{j_{k+1}} u_{i+1}x_1^{-1}$ is a reduced decomposition.  Also $x = x_1 u_{i+1} t_{j_1} \dotsm t_{j_k} t_{j_{k+1}} x_2 u_1 \dotsm u_{i-1} u_{i+1} \dotsm u_l$ is a reduced decomposition.  Inductively, this provides subexpressions $w_1'' \leq w_1'$ and $x_3' \leq x_3$ together with an element $y \in W$ such that
  \begin{itemize}
  \item $y$ commutes with all letters from $t_{j_1} \dotsm t_{j_{k+1}} x_2$, 
  \item $w_1' = w_1'' y$ is a reduced decomposition,
  \item $w = w_1'' t_{j_1} \dotsm t_{j_k} t_{j_{k+1}} y x_1^{-1}$ is a reduced decomposition, 
  \item $x = x_1 y^{-1} t_{j_1} \dotsm t_{j_k} t_{j_{k+1}} x_2 x_3'$ is a reduced decomposition, and 
  \item $|w_1'' t_{j_1} \dotsm t_{j_k} t_{j_{k+1}} x_2x_3'| = |w_1'' t_{j_1} \dotsm t_{j_k} t_{j_{k+1}} x_2| + |x_3'|$.
  \end{itemize}
  It follows that $w_1'' t_{j_1} \dotsm t_{j_k} t_{j_{k+1}} x_2y$ is a reduced expression and so is its initial piece $w_1'' t_{j_1} \dotsm t_{j_k} t_{j_{k+1}} x_2$.  We can conclude that also $w_1'' t_{j_1} \dotsm t_{j_k} t_{j_{k+1}} x_2x_3'$ is a reduced expression.  This finishes the proof.
\end{proof}

In order to formulate consequences of Theorem~\ref{thm:skipping-cancellation-normal-form} for Hecke algebras, we introduce some more notation.
\begin{notation}
  \label{not:product-set-hecke-algebra}
  Let $(W, S)$ be a Coxeter system and $w, x \in W$.  For a deformation parameter $q \in \RR_{>0}^{(W,S)}$, the product set of $w, x$ in the associated Iwahori-Hecke algebra is $\mathrm{Prod}(w,x) = \supp(T_w T_x)$.  The coefficients $p_{u,w,x}$ of the product, are defined by the formula $T_wT_x = \sum_{u \in \mathrm{Prod}(w,x)} p_{u,w,x} T_u$.

  Note that $\mathrm{Prod}(w,x)$ only depends on the set $\{s \in S \mid q_s = 1\}$.  In particular, $\mathrm{Prod}(w,x) = \{w x\}$ for $q \equiv 1$.
\end{notation}

\begin{remark}
  \label{rem:product-set-hecke-algebra-cancellation}
  Consider the setup of Notation~\ref{not:product-set-hecke-algebra}.  By the definition of multiplication in Hecke algebras, for every $w,x \in W$ elements with given reduced expressions $w = s_1 \dotsm s_m$ and $x = t_1 \dotsm t_n$, every element in $\mathrm{Prod}(w,x)$ arises by skipping cancellation at some positions when multiplying the given reduced words.  This makes Theorem~\ref{thm:skipping-cancellation-normal-form} applicable to describe elements of $\mathrm{Prod}(w,x)$.
\end{remark}

The following corollary leverages Remark~\ref{rem:product-set-hecke-algebra-cancellation} in order to give (crude) bounds on the size of product sets in Hecke algebras.
\begin{corollary}
  \label{cor:hecke-multiplication-combinatorics}
  Let $(W, S)$ be a right-angled Coxeter system.  Given $w,x \in W$ the following statements hold true for every deformation parameter $q \in \RR_{> 0}^S$.
  \begin{itemize}
  \item $|\mathrm{Prod}(w,x)| \leq (1 + \min\{|w|, |x|\})^{|S|}$.
  \item If $W$ is additionally hyperbolic, then there is a constant $C > 0$ such that for all $0 \leq n \leq |w| + |x|$ we have $|\{u \in \mathrm{Prod}(w,x) \mid |u| = n\}| \leq C$.
  \item For every $u \in \mathrm{Prod}(w,x)$  we have $p_{u, w,x} \leq \prod_{s \in S} p_s$.
  \end{itemize}
\end{corollary}
\begin{proof}
Fix reduced expressions of length $m = |w|$ and $n = |x|$ for the group elements $w$ and $x$, respectively.  By Remark~\ref{rem:product-set-hecke-algebra-cancellation}, for element $u \in \mathrm{Prod}(w,x)$ there are indices $1 \leq i_1 < \dotsm < i_l \leq n$ such that $u$ arises from from multiplying the reduced expressions for $w$ and $x$, while skipping cancellation at positions $i_1, \dotsc, i_l$.  By Theorem~\ref{thm:skipping-cancellation-normal-form} the letters $t_{i_1}, \dotsc, t_{i_l}$ are pairwise different, as otherwise $w_1 t_{j_1}\cdots t_{j_k} x^{-1}$ could be reduced.  So $u$ is uniquely defined by a map $S \lra \{0, 1, \dotsc, |x|\}$, where a letter in $S$ is mapped to $0$ if it is not skipped.  Hence $|\mathrm{Prod}(w,x)| \leq (1 + |x|)^{|S|}$.  Since $\mathrm{Prod}(w,x) = \mathrm{Prod}(x^{-1}, w^{-1})^{-1}$, we also find $|\mathrm{Prod}(w,x)| \leq (1 + |w^{-1}|)^{|S|} = (1 + |w|)^{|S|}$.  This shows the first claim of the corollary.

  We now assume that $W$ is hyperbolic and prove the second claim.   For $0 \leq n \leq |w| + |x|$ consider $u \in \mathrm{Prod}(w,x)$ satisfing $|u| = |w| + |x| - n$.  By Remark~\ref{rem:product-set-hecke-algebra-cancellation}, we can apply Theorem~\ref{thm:skipping-cancellation-normal-form} to find reduced decompositions $w = w_1 t_1 \dotsc t_k x_1^{-1}$ and $x = x_1 t_1 \dotsc t_k x_2$, such that the elements $t_1, \dotsc, t_k$ commute pairwise and $u = w_1 t_1 \dotsc t_k x_2$ is a reduced decomposition.  It follows that $n = 2|x_1| + k$ and $k \leq |S|$.  We infer that $n/2 - |S| \leq |x_1| \leq n/2 + |S|$.

  Let $n/2 - |S| \leq l \leq n/2 + |S|$ and let $y_1,y_2 \in W$ satisfy $|y_1| = |y_2| = l$ and $|w y_1^{-1}| = |w| - l = |wy_2^{-1}|$.  Let $\delta > 0$ be a hyperbolicity constant for the Cayley graph of $(W, S)$.  Consider the quadrilateral whose corners are $e, wy_1^{-1}, wy_2^{-1}, w$,  whose geodesics connected with $e$ are given by $wy_1^{-1}, wy_2^{-1}$ and whose geodesics connected with $w$ are given by $y_1, y_2$.  Hyperbolicity applied to this quadrilateral shows that
  \begin{gather*}
    |w| + |(wy_1^{-1})^{-1}(w y_2)|
    \leq
    \max\{|wy_1^{-1}| + |y_1|, |wy_2^{-1}| + |y_2|\}  + 2\delta
    =
    |w| + 2\delta
    \eqstop
  \end{gather*}
  So we infer that $|y_1^{-1}y_2| \leq 2 \delta$.  

  Summarising the arguments above, we find that
  \begin{align*}
    & |\{u \in \mathrm{Prod} \mid |u| = |w| + |x| - n\}|  \\
     \leq &
    |\{y \in W \mid \frac{n}{2} - |S| \leq |y| \leq \frac{n}{2} + |S| \text { and } |wy^{-1}| = |w| - |y|\}| \\
    \leq &
           (2|S| + 1) |\rB_{2\delta}(e)|
    \eqstop
  \end{align*}
  So the second statement follows for $C = (2 |S| + 1) |\rB_{2\delta}(e)|$.

  Let us prove the last claim of the corollary.  Fix $u \in \mathrm{Prod}(w,x)$.  By Theorem~\ref{thm:skipping-cancellation-normal-form} there are reduced decompositions $w = w_1 t_1\dotsm t_l x_1^{-1}$ and $x = x_1 t_1 \dotsm t_l x_2$ such that $t_1, \dotsc, t_l \in S$ commute pairwise and $u = w_1t_1\dotsm t_l x_2$ is a reduced decomposition.  So by definition of the Iwahori-Hecke algebra multiplication, we obtain $p_{u,w,x} = p_{t_1} \dotsm p_{t_l} \leq \prod_{s \in S} p_s$.
\end{proof}

The following result connects cyclic reduction in right-angled Coxeter groups with the associated product in its Hecke algebra.
\begin{proposition}
  \label{prop:minimal-representatives-hecke-conjugation}
  Let $(W, S)$ be a right-angled Coxeter system and let $w \in W$.  Assume that  $w = x w' x^{-1}$ (for some $x,w' \in W$) is a reduced decomposition such that $w'$ is of minimal length in its conjugacy class.  Then every element in $\mathrm{Prod}(x^{-1}, x w')$ is of minimal length in its conjugacy class.
\end{proposition}
\begin{proof}
  Observe similarly to Remark~\ref{rem:product-set-hecke-algebra-cancellation} that every $u \in \mathrm{Prod}(x^{-1}, xw')$ has a reduced decomposition $u = u'w'$ for $u' \in \mathrm{Prod}(x^{-1}, x)$.  By Theorem~\ref{thm:skipping-cancellation-normal-form}, there are reduced decompositions $x^{-1} = y t_1 \dotsm t_k z^{-1}$ and $x = z t_1 \dotsm t_k a$ such that the letters $t_1, \dotsc, t_k$ commute pairwise, $u' = y t_1 \dotsm t_k a$ is a reduced decomposition and $ya = x^{-1}x = e$ as elements of $W$.  It follows that $a = y^{-1}$.

Suppose that $u$ was not of minimal length in its conjugacy class.  Then  Theorem~\ref{thm:cyclic-reduction-shifts-right-angled} due to Marquis says that $u$ is not cyclically reduced.  So it suffices to consider $s \in S$ which is an initial and a terminal letter of $u$, and prove that it commutes with all elements of $S(u)$.  Record for use in the proof that all subexpressions of consecutive letters of the reduced expression $x w' x^{-1} = z t_1 \dotsm t_k y^{-1} w' y t_1 \dotsm t_k z^{-1}$ are reduced.  In particular, the expressions $w'y$, $y^{-1} w' y$ and $w' y t_1 \dotsm t_k$ are reduced.  We now consider the first occurrence of $s$ in the reduced decomposition $u = u'w' = y t_1 \dotsm t_k y^{-1} w'$.

Assume first that $s$  is a letter of $y$.  Then it automatically is an initial letter of $y$.  If $s$ was a terminal letter of $w'$, we would obtain a contradiction to the fact that $w'y$ is reduced.  It follows that $s$ must commute with $w'$, which now contradicts the fact that the expression $y^{-1} w' y$ is reduced.  In conclusion, $s$ is not a letter of $y$.

Assume then that $s$ appears in $t_1, \dotsc, t_k$.  Then $s$ commutes with all letters $t_1, \dotsc, t_k$ and with $y$.  So it commutes with $y t_1 \dotsm t_k$.  Hence, it can only appear once in the reduced decomposition $u' = y t_1 \dotsm t_k y^{-1}$.  So $s$ must be a terminal letter of $w'$.  But this contradicts the fact that $w' y t_1 \dotsm t_k$ is reduced.  We infer that $s$ must be an initial and terminal letter of $w'$ and commute with $y t_1 \dotsm t_k y^{-1}$.  Since $w'$ is cyclically reduced, it follows that $s$ also commutes with $w'$. So it commutes with all of $u$.  This is what we had to show.
\end{proof}

\section{Smooth subalgebras revisited}
\label{sec:smooth-subalgebras-revisited}

We refer to Bost's work \cite[Appendix A]{bost1990-principe-oka} for a reference on smooth subalgebras.  Let us recall their definition before we outline the content of this secton.
\begin{definition}
  \label{def:smooth-subalgebra}
  Let $A$ be unital C*-algebra.  We say that a subset $\cA \subseteq A$ is a smooth subalgebra if
  \begin{itemize}
  \item $\cA$ is a unital Frech{\'e}t algebra whose group of invertible elements is open and such that inversion is a continuous map on $\mathrm{GL}(\cA)$,
  \item the inclusion map $\cA \hra A$ is continuous,
  \item $\cA$ is dense in $A$, and
  \item $\cA$ is inverse-closed, that is if $a \in \cA$ is invertible in $A$, then $a^{-1} \in \cA$.
  \end{itemize}
\end{definition}

We introduce an abstract framework capturing the classical argument of Connes \cite{connes1986} and Jolissaint \cite{jolissaint1989-k-theory} providing smooth subalgebras.
\begin{definition}
  \label{def:empowered-differential-seminorm}
  Let $(\cA, \| \cdot \|)$ be a unital, C*-normed *-algebra.  An empowered differential seminorm on $\cA$ is a map $\rD\colon \cA \lra \rmc(\NN_0, \RR_{\geq 0})$  such that
  \begin{itemize}
  \item[(i)] each component $\rD_n\colon \cA \lra \RR_{\geq 0}$ is a seminorm,
  \item[(ii)] there is $C > 0$ satisfying $\rD_0(a) \leq C \|a\|$ for all $a \in \cA$,
  \item[(iii)] for every $k \in \NN_0$ there is a polynomial $f_k \in \RR[(X_n)_{n \in \NN_0}, (Y_n)_{n \in \NN_0}, U, V]$ with zero constant coefficient such that for all $a,b \in \cA$ we have
    \begin{gather*}
      \rD_k(ab) \leq f_k((\rD_n(a))_{n \in \NN_0}, (\rD_n(b))_{n \in \NN_0}, \|a\|, \|b\|)
      \eqcomma \text{ and }
    \end{gather*}
  \item[(iv)] there is $C' > 0$ such that for every $k \in \NN_0$ there is a polynomial $g_k \in \RR[(X_n)_{n \in \NN_0}, Y]$ with zero constant coefficient and some function $h_k\colon \NN_0 \lra \RR_{\geq 0}$ vanishing at infinity such that for all $1 \leq M \leq N$ and for all $a \in \cA$ satisfying $\|a\| < C$ we have
    \begin{gather*}
      \rD_k(\sum_{l = M}^Na^l) \leq h_k(M) \cdot g_k((\rD_n(a))_{n \in \NN_0}, \|a\|) 
      \eqstop
    \end{gather*}
  \end{itemize}
\end{definition}

\begin{remark}
  \label{rem:empowered-differential-seminorm-frechet-algebra}
  If $\cA$ is a unital, C*-normed *-algebra with an empowered differential seminorm $\rD$, then the entries $\rD_n(a) = \rD(a)_n$ for $n \in \NN_0$, $a \in \cA$, define a family of seminorms such that the separation-completion of $\cA$ with respect to $(\rD_n)_n$ and $\| \cdot \|$ is a Frech{\'e}t algebra.  Indeed, the inequality
  \begin{gather*}
    \rD_k(ab) \leq f_k((\rD_n(a))_{n \in \NN_0}, (\rD_n(b))_{n \in \NN_0}, \|a\|, \|b\|)
  \end{gather*}
  for all $a,b \in \cA$ shows continuity of the multiplication on the resulting Frech{\'e}t space.  Often we will assume that there are $C > 0$ and $N \in \NN_0$ such that $\|a\| \leq C \rD_N(a)$ for all $a \in \cA$.  In this case, we can equivalently consider the completion solely with respect to the family $(\rD_n)_n$.
\end{remark}

In Section~\ref{sec:smooth-subalgebras-hecke} the following proposition will be used to find empowered differential seminorms, by proving parts of the required inequalities for a family of seminorms which is equivalent in a suitable sense to the desired one.
\begin{proposition}
  \label{prop:empowered-differential-seminorm-equivalence}
  Assume that $\cA$ is a unital, C*-normed *-algebra with two maps $\rD, \rD'\colon \cA \lra \rmc(\NN_0, \RR_{\geq 0})$ satisfying the following properties.
  \begin{itemize}
  \item[(i)] Each component $\rD_n$ is a seminorm on $\cA$,
  \item[(ii)] there is $C > 0$ satisfying $\rD(a)_0 \leq C \|a\|$ for all $a \in \cA$,
  \item[(iii)] for every $k \in \NN_0$ there is a polynomial $f_k \in \RR[(X_n)_{n \in \NN_0}, (Y_n)_{n \in \NN_0}, N, M]$ with zero constant coefficient such that for all $a,b \in \cA$ we have
    \begin{gather*}
      \rD_k(ab) \leq f_k((\rD_n(a))_{n \in \NN_0}, (\rD_n(b))_{n \in \NN_0}, \|a\|, \|b\|)
      \eqcomma
    \end{gather*}
  \item[(iv')] there is $C > 0$ such that for every $k \in \NN_0$ there is a polynomial $g_k \in \RR[(X_n)_{n \in \NN_0}, Y]$ with zero constant coefficient and some function $h_k\colon \NN_0 \lra \RR_{\geq 0}$ vanishing at infinity such that for all $1 \leq M \leq N$ and for all $a \in \cA$ satisfying $\|a\| < C$ we have
    \begin{gather*}
      \rD'_k(\sum_{l = M}^Na^l) \leq h_k(M) \cdot g_k((\rD'_n(a))_{n \in \NN_0}, \|a\|) 
      \eqcomma \text{ and }
    \end{gather*}
  \item[(v)] for every $k \in \NN_0$ there are polynomials with zero constant coefficient $p_k,q_k \in \RR[(X_n)_n, Y]$ such that
    \begin{gather*}
      \rD_k(a) \leq p_k((\rD'_n(a))_{n \in \NN_0}, \|a\|)
      \quad
      \text{and}
      \quad
      \rD'_k(a) \leq q_k((\rD_n(a))_{n \in \NN_0}, \|a\|)
      \qquad
      \text{for all } a \in \cA
      \eqstop
    \end{gather*}
  \end{itemize}
  Then  $\rD$ is an empowered differential seminorm.
\end{proposition}
\begin{proof}
  We have to find $C > 0$ allowing us to provide a suitable estimate for $\rD_k(\sum_{l = M}^Na^l)$ for all $k \in \NN_0$, all $1 \leq M \leq N$ and all $a \in \cA$ satisfying $\|a\| < C$.  Let $C > 0$ be a constant such that there are $h_k$ and $g_k$ as in (iv') satisfying
  \begin{gather*}
    \rD'_k(\sum_{l = M}^Na^l) \leq h_k(M) \cdot g_k((\rD'_n(a))_{n \in \NN_0}, \|a\|) 
    \eqcomma 
  \end{gather*}
  for all $1 \leq M \leq N$ and all $a \in \cA$ satisfying $\|a\| < C$.  We may assume that $p_k$ as in the statement of the proposition has non-negative coefficients, so that the associated polynomial function is monotone.  Write $p_k = p_{k,1} + p_{k,2}$ with $p_{k,1} \in \sum_m X_m\RR[(X_n)_n, Y]$ and $p_{k,2} \in \RR[Y]$ with zero constant coefficient.  Then 
    \begin{align*}
      \rD_k(\sum_{l = M}^Na^l)
      & \leq
      p_k((\rD'_n(\sum_{l = M}^Na^l))_{n \in \NN_0}, \|\sum_{l = M}^Na^l\|) \\
      & =
      p_{k,1}((\rD'_n(\sum_{l = M}^Na^l))_{n \in \NN_0}, \|\sum_{l = M}^Na^l\|) + p_{k,2}(\|\sum_{l = M}^N a^l\|)
      \eqstop
    \end{align*}
   As $p_{k,2}$ has no constant coefficient, using $\|a\| \leq C < 1$, we obtain that
    \begin{gather*}
      p_{k,2}(\|\sum_{l = M}^N a^l\|)
      \leq
      \tilde C \|a\|^M
      =
      \tilde C \|a\|^{M-1} \|a\|
    \end{gather*}
    for some $\tilde C > 0$.

    The other summand can be bounded as follows:
    \begin{gather*}
      p_{k,1}((\rD'_n(\sum_{l = M}^Na^l))_{n \in \NN_0}, \|\sum_{l = M}^Na^l\|)
      \leq
      p_{k,1}((h_n(M) g_n((\rD'_m(a))_{m \in \NN_0}, \|a\|))_{n \in \NN_0}, \|\sum_{l = M}^Na^l\|)
      \eqstop
    \end{gather*}
    Since $\|\sum_{l = M}^Na^l\| \leq \sum_{l = M}^N \|a\|^l \leq \frac{1}{1 - C}$ and every monomial appearing in $p_{k,1}$ is divisible by some $X_n$, we find that
    \begin{gather*}
      p_{k,1}((\rD'_n(\sum_{l = M}^Na^l))_{n \in \NN_0}, \|\sum_{l = M}^Na^l\|)
      \leq
      \tilde h(M) \tilde g_k((\rD'_m(a))_{m \in \NN_0}, \|a\|)
    \end{gather*}
    for the function $\tilde h = \max_{n \leq \deg p_{k,1}} h_n$ vanishing at infinity and some polynomial $\tilde g_k \in \RR[(X_n)_{n \in \NN_0}, Y]$.  Using the estimate $\rD'_m(a) \leq q_m((\rD_i(a))_{i \in \NN_0}, \|a\|)$, we obtain an estimate on $\rD_k(\sum_{l = M}^Na^l)$ needed to show that $\rD$ is an empowered differential seminorm.
\end{proof}

\begin{theorem}
  \label{thm:smooth-subalgebra-differential-seminorm}
  Let $D\colon \cA \ra \rmc(\NN_0, \RR_{\geq 0})$ be an empowered differential seminorm on a unital, C*-normed *-algebra $\cA$. Assume that there is $C > 0$ and $N \in \NN_0$ such that $\|a\| \leq C \rD_N(a)$ for all $a \in \cA$.  Denote by $A$ the C*-algebra closure of $\cA$ and by $\cB$ its completion with respect to the family of seminorms $(D_n)_n$.  Then $\cA \subseteq A$ extends to a continuous inclusion $\cB \lra A$ whose image is a smooth subalgebra.
\end{theorem}
\begin{proof}
  Let $C > 1$ and $N \in \NN_0$ be chosen to satisfy $\|a\| \leq C \rD_N(a)$ and $\rD_0(a) \leq C \|a\|$ for all $a \in \cA$.  The first bound implies that every Cauchy sequence for the uniformity defined by $\rD$ is a Cauchy sequence for the C*-norm on $\cA$.  Hence, we obtain a continuous inclusion $\cB \lra A$.  We may thus assume that $\cA$ is complete in the uniformity defined by $\rD$.  We have to show that under this identification $\cA$ is a smooth subalgebra of $A$.  It is clear that $\cA$ is a unital, dense Fr{\'e}chet subalgebra of $A$.  So it remains to show that $\mathrm{GL}(\cA) \subseteq \cA$ is open, that inversion is continuous and that $\cA$ is inverse-closed in $A$.

 Let us first show that the subset of invertible elements $\mathrm{GL}(\cA) \subseteq \cA$ is open.  We may increase $C$ so that $C^{-1}$ and polynomials $g_k \in \RR[(X_k)_{k \in \NN_0}, Y]$ witness the last part of Definition~\ref{def:empowered-differential-seminorm} of an empowered differential seminorm.  Let $a \in \cA$ be invertible and assume that $b \in \cA$ satisfies $\|a - b\| < (C\|a^{-1}\|)^{-1}$.  Let us show that $b \in \mathrm{GL}(\cA)$.  We put $c = 1 - ba^{-1}$.  Then we have $\|c\| = \|1 - ba^{-1}\| \leq \|a - b\| \|a^{-1}\| < C^{-1}$.  By the choice of $C$, it follows that $\sum_{l = 1}^\infty c^l$ is summable in every seminorm $\rD_n$.  So $(1 -c)^{-1} = \sum_{l = 0}^\infty c^l$ exists in $\cA$.  As $1 - c = ba^{-1}$, it follows that $b$ is right-invertible in $\cA$.  Arguing the same way for $1 - a^{-1}b$, we infer that $b$ is also left-invertible in $\cA$.  It hence follows that $b \in \mathrm{GL}(\cA)$.  We now conclude that
  \begin{gather*}
    \{b \in \cA \mid \rD_n(b - a) \leq (C^2\|a^{-1}\|)^{-1} \}
    \subseteq
    \{b \in \cA \mid \|b - a\| \leq (C\|a^{-1}\|)^{-1} \}
    \subseteq
    \mathrm{GL}(\cA)
  \end{gather*}
  is an open neighbourhood of $a$.  So $\mathrm{GL}(\cA) \subseteq \cA$ is open.

  Let us now show that inversion is continuous on $\mathrm{GL}(\cA)$.  For a convergent sequence $a_k\to a$ in $\mathrm{GL}(\cA)$, we have $a_k^{-1} = a^{-1} (a_ka^{-1})^{-1}$.  As $a_ka^{-1} \to 1$, it suffices to prove continuity of inversion at $1$.  Let $b \in \cA$ satisfy $\|1 - b\| < C^{-1}$.  For $c = 1 - b$, we have seen that $b^{-1} = \sum_{l = 0}^\infty c^l \in \cA$.  So $b^{-1} - 1 = \sum_{l = 1}^\infty c^l$.  Furthermore, since $\rD$ is an empowered differential seminorm, we find that $\rD_n(\sum_{l = 1}^\infty c^l) \leq g_n((\rD_k(c))_k)$ for some polynomials $g_n$ without constant coefficient.  In particular, we infer that $b^{-1} \to 1$ as $b \to 1$.

  It remains to show that $\cA \subseteq A$ is inverse-closed.  We proceed along the lines of standard arguments for Banach algebras.  We first show that for every $a \in \cA$ the spectral radius satisfies $\rmr_A(a) = \rmr_{\cA}(a)$.  Since every element that is invertible in $\cA$ is also invertible in $A$, it follows that $\sigma_A(a) \subseteq \sigma_\cA(a)$ and hence $\rmr_A(a) \leq \rmr_{\cA}(a)$.  To prove the converse implication, we may rescale $a$, so that it suffices to prove the following property: whenever $\rmr_A(a) < 1$, then $\sigma_{\cA}(a) \cap \rS^1 = \emptyset$ holds.  Fix $z \in \rS^1$.  The Beurling formula for the spectral radius in $A$ shows that $\|a^k\| \to 0$ as $k \to \infty$.  So there are $k \in \NN$ such that $\|a^k\| < C^{-1}$.  Put $b = \frac{1}{z^k}a^k$.  Then $1 - b$ is invertible, as $\|1 - (1 - b)\| = \|b\| < C^{-1}$ and $D$ is an empowered differential seminorm.  So $1 \notin \sigma_{\cA}(b)$, which implies that $z^k \notin \sigma_{\cA}(a^k)$.  It then follows that $z \notin \sigma_{\cA}(a)$ holds. 
So we have shown that $\rmr_A(a) = \rmr_\cA(a)$ for all $a \in \cA$.

Next, we will show that $\sigma_A(a) = \sigma_\cA(a)$ for all $a \in \cA$.  Assume for a contradiction that $\sigma_A(a) \neq \sigma_{\cA}(a)$, that is $\sigma_A(a) \subsetneq \sigma_{\cA}(a)$.  Considering a small ball in $\CC$ that intersects $\sigma_{\cA}(a)$ but not $\sigma_A(a)$, we can find a function $f\colon \CC \lra \CC$ which is a composition of a reflection along the boundary of this ball with a rigid motion, such that $f(\sigma_A(a)) \subseteq \rB_1(0)$ but there is some $\lambda \in f(\sigma_\cA(a))$ satisfying $\Re \lambda < -1$.  We may replace $f$ by some sufficiently good polynomial approximation satisfying the same properties.  Then we have $f(\sigma_{\cA}(a)) \subseteq \sigma_\cA(f(a))$ 
and by the spectral mapping theorem for C*-algebras we have $f(\sigma_A(a)) = \sigma_A(f(a))$.  We infer that $\rmr_A(f(a)) < 1 < \rmr_{\cA}(f(a))$, contradicting what we have already shown.  It follows that the spectrum of $a \in \cA$ agrees with its spectrum in $A$.  In particular $\cA \subseteq A$ is inverse-closed.
\end{proof}

In \cite{blackadarcuntz1991,puschnigg2010} the concept of quasi-derivations for norms was employed.  A map $\Delta\colon \cA \lra \cA \otimes \cA$ on a normed algebra $\cA$ with fixed cross-norm on $\cA \otimes \cA$ was called a quasi-derivation if there is $C > 0$ such that for all $a,b \in \cA$ the inequality
\begin{gather*}
  \|\Delta(ab)\| \leq C (\|a\| \|\Delta(b)\| + \|\Delta(a)\|\|b\|)
\end{gather*}
holds.  We adapt this notion to the setup of empowered differential seminorms.
\begin{definition}
  \label{def:quasi-derivation}
  Let $\rD$ be a empowered differential seminorm on a unital, C*-normed *-algebra $\cA$.  A quasi-derivation with constant $C > 0$ on $(\cA, \rD)$ is a continuous, linear map $\delta\colon \cA \lra \cM$ into a Fr{\'e}chet space with a generating family of seminorms $(\rD^\cM_n)_n$ such that there is $C > 0$ satisfying
  \begin{gather*}
    \rD_n^\cM(\delta(ab)) \leq C(\rD_0(a)\rD_n^\cM(\delta(b)) + \rD_n^\cM(\delta(a))\rD_0(b))
    \qquad
    \text{ for all } a,b \in \cA, n \in \NN_0
    \eqstop
  \end{gather*}
\end{definition}

We now have an analogue of \cite[Proposition 5.4]{blackadarcuntz1991} and \cite[Theorem 3.8 (c)]{puschnigg2010}.
\begin{proposition}
  \label{prop:quasi-derivation-gives-differential-seminorm}
  Let $\cA$ be a unital, C*-normed *-algebra with empowered differential norm $\rD$.  Let $C \geq 1$ and assume that $\Delta\colon \cA \lra \cM$ is a quasi-derivation with constant $C$ for every component of $\rD$.  Then the formulae
  \begin{gather*}
    \rD'_0(a) = C \rD_0(a)
    \eqcomma \qquad
    \rD_n'(a) = \rD_n(a) + \rD^\cM_{n-1}(\Delta(a))
    \quad \text{for all } a \in \cA, n \geq 1
  \end{gather*}
  define an empowered differential norm on $\cA$.  Writing $\cB$ and $\cB'$ for the completions of $\cA$ with respect to $\rD$ and $\rD'$, the identity on $\cA$ extends to a continuous inclusion $\cB' \hra \cB$.
\end{proposition}
\begin{proof}
  We shall first verify that the objects defined in the proposition indeed satisfy all properties of an empowered differential seminorm.  Given $C_0 > 0$ which satisfies $\rD_0(a) \leq C_0 \|a\|$ for all $a \in \cA$, we have $\rD'_0(a) \leq CC_0\|a\|$.  It is clear that $\rD'$ is positive homogeneous and subadditive.  For $a,b \in \cA$ we have 
  \begin{gather*}
    \rD'_0(ab) = C \rD_0(ab) \leq C \rD_0(a)\rD_0(b) \leq \rD'_0(a)\rD'_0(b)
    \eqstop
  \end{gather*}
  For $k \geq 1$ let $f_k$ be as in Definition~\ref{def:empowered-differential-seminorm} of an empowered differential seminorm.  We obtain
  \begin{align*}
    \rD'_k(ab)
    & =
      \rD_k(ab) + \rD^\cM_{k-1}(\Delta(ab)) \\
    & \leq
      f_k((\rD_n(a)), (\rD_n(b)), \|a\|, \|b\|) + C \left ( \rD_0(a)\rD^\cM_{k-1}(\Delta(b)) + \rD^\cM_{k-1}(\Delta(a))\rD_0(b) \right ) \\
    & \leq
      f_k((\rD'_n(a)), (\rD'_n(b)), \|a\|, \|b\|) + C \left ( \rD'_0(a)\rD'_k(b) + \rD'_k(a)\rD'_0(b) \right )
      \eqstop
  \end{align*}
  The last expression is bounded by a polynomial with zero constant coefficient in $(\rD'_n(a))_n$, $(\rD'_n(b))_n$, $\|a\|$ and $\|b\|$.

  Now let $C' > 0$ and $g_k \in \RR[(X_n)_n , Y]$ be as in the last item of Definition~\ref{def:empowered-differential-seminorm}.  We may assume that $C' < 1$ and $C' \leq \frac{1}{2}C^{-1}$.  If $a \in \cA$ satisfies $\|a\| \leq C'$, then
  \begin{align*}
    \rD'_k(\sum_{l = M}^N a^l)
    & =
      \rD_k(\sum_{l = M}^N a^l) + \rD^\cM_{k-1}(\Delta(\sum_{l = M}^N a^l)) \\
    & \leq
      h(M) g_k(\rD_n(a), \|a\|) + C \sum_{l = M}^N C^{l-1} \|a\|^{l-1} \rD^\cM_{k-1}(\Delta(a)) \\
    & \leq
      h(M) g_k(\rD_n(a), \|a\|) + C (C \|a\|)^{M} \sum_{l = 0}^\infty (C \|a\|)^{l-1}  \rD^\cM_{k-1}(\Delta(a))
      \eqstop
  \end{align*}
  Since $C \|a\| \leq \frac{1}{2}$, the last term is a product of a function in $M$ that vanishes at infinity and a polynomial with zero constant coefficient in $(\rD_n(a))_n$, $\|a\|$ and $\rD^\cM_{k-1}(\Delta(a))$.  Since $\rD_n(a) \leq \rD'_n(a)$ and $\rD^\cM_{k-1}(\Delta(a)) \leq \rD'_{k-1}(a)$ for all $a\in \cA$ and for all $n\in \NN_0$, the term is bounded by a polynomial with zero constant coefficient in $(\rD'_n(a))_n$ and $\|a\|$.
   
  We have verified all defining properties of an empowered differential seminorm.  The fact that the identity on $\cA$ extends to a continuous map $\cB' \lra \cB$ follows from the inequality $\rD_n(a) \leq \rD'_n(a)$, which holds for all $n \in \NN_0$ and for all $a \in \cA$.  Indeed, it implies that every Cauchy sequence for the uniformity defined by $\rD'$ is a Cauchy sequence for the uniformity defined by $\rD$.  The map $\cB' \lra \cB$ obtained so is injective, since $\Delta$ is closable with respect to any component of $\rD$, as follow precisely by the same arguments as \cite[Theorem 3.8 (a)]{puschnigg2010}.  Let us give some details.  Write $\vphi\colon \cB' \lra \cB$ for the continuous map which satisfies $\vphi|_\cA = \id_\cA$.  Let $b \in \ker \vphi$.  We have to show $b = 0$.  Let $(a_n)_n$ be a sequence in $\cA$ converging to $b$ in $\rD'$.  Then $(a_n)_n$ is Cauchy in $(\cA, \rD)$ and its limit is $0 \in \cB$, as $0 = \vphi(a) = \lim_\rD \vphi(a_n) = \lim_\rD a_n$.  Further, as $(a_n)_n$ converges in $\rD'$, the sequence $(\Delta(a_n))_n$ is Cauchy in $\rD^\cM$. Let $m \in \cM$ be its limit.  Then $(a_n, \Delta(a_n)) \lra (0,m)$ in the graph uniformity of $\Delta$.  As $\Delta$ is closable, we infer that $m = 0$.  This shows that $\rD'_k(b) = \lim_n \rD'_k(a_n) = \lim_n \rD_k(a_n) + \rD^\cM_k(\Delta(a_n)) = 0$.  So $b = 0$ follow from the fact that $\rD'$ is a differential norm.
\end{proof}

\section{Smooth subalgebras of Hecke algebras}
\label{sec:smooth-subalgebras-hecke}

In this section we combine results on the combinatorics of multiplication in right-angled Hecke algebras from Section~\ref{sec:combinatorics-hecke-algebra} with the notion of empowered differential seminorms developed in Section~\ref{sec:smooth-subalgebras-revisited}.  This way, we will find specific smooth subalgebras of reduced Hecke C*-algebras.

The following definition applies Bost-Lafforgue's notion of unconditional seminorms for group algebras (see \cite[Definition 2.1]{puschnigg2010}) to the setup of Hecke algebras. 
\begin{definition}
  \label{def:unconditional-seminorm-hecke}
  Let $(W, S)$ be a Coxeter system and $q \in \RR_{> 0}^{(W,S)}$ a deformation parameter.  Then an empowered differential seminorm on $\CC(W, S, q)$ is called unconditional if its components are unconditional with respect to the basis $(T_w)_{w \in W}$.  A Fr{\'e}chet algebra completion of $\CC(W, S, q)$ is called unconditional, if its topology is induced by a family of seminorms on $\CC(W, S, q)$ which are unconditional with respect to $(T_w)_{w \in W}$.
\end{definition}

Let us define the natural analogue of the Sobolev norms on the group algebra of a marked group (i.e. a group with a fixed generating set) in the setting of Hecke algebras.
\begin{definition}
  \label{def:sobolev-norms}
  Let $(W, S)$ be a Coxeter system and $q \in \RR_{>0}^{(W, S)}$ a deformation parameter.  For $s \in \RR_{\geq 0}$ we define as usual the associated (2,$s$)-Sobolev norm $\|\cdot\|_{2,s}$: for $a = \sum_w a_w T_w \in \CC(W, S, q)$ set
  \begin{gather*}
    \|a\|_{2,s} = \bigl ( \sum_w |a_w|^2 (1 + |w|)^{2s} \bigr )^{1/2}
    \eqstop
  \end{gather*}
  The completion of $\CC(W, S, q)$ with respect to the family of seminorms $(\| \cdot \|_{2,s})_{s \geq 0}$ is denoted by $\ltwo_{\mathrm{RD}}(W, S, q)$.
\end{definition}

\begin{remark}
  \label{rem:sobolev-norm-monotone}
  The Sobolev norms are monotone in the sense the $\|a\|_{2,s} \leq \|a\|_{2,t}$ holds for all $s \leq t$ and all $a \in \CC(W, S, q)$.  As a consequence, $\ltwo_{\mathrm{RD}}(W, S, q)$ is also the completion of $\CC(W, S, q)$ with respect to the family $(\| \cdot \|_{2,s})_{s \in \NN_0}$.
\end{remark}

\begin{remark}
  \label{rem:rd-algebra-cstar-algebra}
  As the (2,0)-Sobolev norm is just the 2-norm on $\CC(W, S, q)$, it follows that there is an injective map $\ltwo_{\mathrm{RD}}(W, S, q) \hra \ltwo(W)$.  We identify $\ltwo_{\mathrm{RD}}(W, S, q)$ with its image under this map.  Likewise, we consider $\Cstarred(W, S, q)$ a subset of $\ltwo(W)$ by the map $a \mapsto a \delta_e$.  This allows us to compare $\ltwo_{\mathrm{RD}}(W, S, q)$ and $\Cstarred(W, S, q)$ which is frequently done in the setup of groups.  See e.g.\ \cite{jolissaint1989-k-theory}.
\end{remark}

Let us record a consequence of the work of Caspers-Klisse-Larsen \cite{caspersklisselarsen2021} on the rapid decay property for right-angled Hecke algebras (called a Haagerup inequality there).  It is the analogue of classical results of Jolissaint's in the setting of reduced group C*-algebras \cite{jolissaint1986} and  the proof follows the arguments in this reference, using \cite{caspersklisselarsen2021} in lieu of \cite{haagerup1979}.
\begin{proposition}
  \label{prop:inclusion-rd-algebra}
  Let $(W, S)$ be a right-angled Coxeter system and $q \in \RR_{>0}^S$ a deformation parameter.  Then the following statements hold true.
  \begin{enumerate}
  \item There is $C > 0$ such that $\|a\| \leq C \|a\|_{2,2}$ for all $a \in \CC(W, S, q)$.
  \item There is $t \geq 0$ such that for all $s \geq 0$ and all $a, b \in \CC(W, S, q)$ we have
    \begin{gather*}
      \|ab\|_{2,s} \leq \|a\|_{2,s} \|b\|_{2, s + t}
      \eqstop
    \end{gather*}
    In particular, there is a unique *-algebra structure on $\ltwo_{\mathrm{RD}}(W, S, q)$ which continuously extends the *-algebra structure on $\CC(W, S, q)$.
  \item We have $\ltwo_{\mathrm{RD}}(W, S, q) \subseteq \Cstarred(W, S, q)$.
  \end{enumerate}
\end{proposition}
\begin{proof}
  We prove the first item by translating work of Caspers-Klisse-Larsen.  Indeed, Caspers-Klisse-Larsen's \cite[Theorem 3.4]{caspersklisselarsen2021} says that every  $a \in \CC(W, S, q)$ supported on the sphere of radius $d$ satisfies
  \begin{gather*}
    \|a\| \leq d |\mathrm{Cliq}(\Gamma_W)|^3 (\prod_{s \in S} p_s) \|a\|_2
    \eqcomma
  \end{gather*}
  where $\Gamma_W$ is the commutation graph of $(W,S)$, and $p_s$ was defined in Subsection \ref{sec:hecke-algebras}.  Arguing exactly as in Haagerup's \cite[Lemma 1.5]{haagerup1979}, we thus obtain
  \begin{gather*}
    \|a\| \leq 2 |\mathrm{Cliq}(\Gamma_W)|^3 (\prod_{s \in S} p_s) \|a\|_{2,2}
    \eqcomma
  \end{gather*}
  for all $a \in \CC(W, S, q)$.  This proves the first item with $C = 2 |\mathrm{Cliq}(\Gamma_W)|^3 (\prod_{s \in S} p_s)$.  The remainder of the proof makes use of this estimate combined with the linear structure of $\ltwo_{\mathrm{RD}}(W, S, q)$ and as such follows verbatim as in \cite[in particular Lemma 1.2.4 and Proposition 1.2.6]{jolissaint1986}.
\end{proof}

The Sobolev norms are clearly unconditional and provide a candidate for the structure of an empowered differential seminorm on right-angled Hecke algebras.  Verifying the remainder of the definition directly in terms of Sobolev norms turns out to be difficult, whence we introduce the following notion, based on Connes' \cite{connes1986} and Jolissaint's \cite{jolissaint1989-k-theory} work.
\begin{definition}
  \label{def:propagation-seminorm}
  Let $(W, S)$ be a right-angled Coxeter systems and $q \in \RR_{>0}^S$ a deformation parameter.  For real numbers $n \geq 0$, we denote by $P_n \in \bo(\ltwo(W))$ the orthogonal projection onto the closed subspace $\cspan\{ \delta_w \mid |w| \leq n\}$; note that in fact $P_n = P\,_{\lfloor n \rfloor}$.  For $r \geq 1$ and $\alpha \in (0,1)$, we define the $(r,\alpha)$-propagation seminorm by 
  \begin{gather*}
    \nu_{r,\alpha}(a) = \sup_{n \geq 1} \Bigl ( n^r \bigl (\|(1 - P_n) a (P_{n - n^\alpha})\| + \|(P_{n - n^\alpha}) a (1 - P_n)\| \bigr ) \Bigr ) 
  \end{gather*}
  for $a \in \Cstarred(W, S, q)$.
\end{definition}

\begin{remark}
  \label{rem:sobolev-propagation-seminorm}
  On their domain of definition, the propagation seminorms are just pseudo-seminorms, as they can take the value $\infty$.  For the purpose of comparing Sobolev norms and propagation seminorms, it is practical to consider this framework, as it allows us to compare both on arbitrary elements of a Hecke C*-algebras $\Cstarred(W, S, q) \subseteq \ltwo(W)$.
\end{remark}

The next proof follows the arguments of \cite[Proposition 1.3]{jolissaint1989-k-theory} adapted to the setup of Hecke C*-algebras.
\begin{proposition}
  \label{prop:sobolev-norm-comparison}
  Let $(W, S)$ be a right-angled Coxeter system and $q \in \RR_{>0}^S$ a deformation parameter.
  \begin{enumerate}
  \item For every $n \geq 0$ there is $C > 0$ such that
    \begin{gather*}
      \| a \|_{2,n} \leq C (\nu_{n+1,1/2}(a) + \|a\|) \qquad \text{ for all } a \in \Cstarred(W, S, q)
      \eqstop
    \end{gather*}
  \item For every $r \geq 1$, $\alpha \in (0,1)$ there is $n \geq 0$ and $C > 0$ such that
    \begin{gather*}
      \nu_{r,\alpha}(a) \leq C \| a \|_{2,n}  \qquad \text{ for all } a \in \ltwo_{\mathrm{RD}}(W, S, q)
      \eqstop
    \end{gather*}
  \end{enumerate}
\end{proposition}
\begin{proof}
  It is clear that $\|\cdot\|_{2,0} = \|\cdot\|_2 \leq \|\cdot\|$.  So let $n \geq 1$ and $a \in \Cstarred(W, S, q)$.  By Remark~\ref{rem:sobolev-norm-monotone}, we may assume that $n \in \NN_0$.  We write $a = \sum_w a_w T_w$, with the series convergent in $\ltwo(W)$.  For all $k \geq 1$, we have
  \begin{align*}
    \sum_{|w| > k} |a_w|^2
    =
    \|(1 - P_k)a\delta_e\|_2^2
    =
    \|(1 - P_k)a P_{k - k^{1/2}} \delta_e\|_2^2
    \leq
    k^{-2r} \nu_{r, 1/2}(a)^2
    \eqstop
  \end{align*}
  Let us note an auxiliary estimate for arbitrary $m \geq 2$.  We have
  \begin{align*}
    (1 + m)^{2n}
    & =
      2n \int_0^m (1 + t)^{2n - 1} \rmd t + 1\\
    & \leq
      4n \sum_{k = 1}^m (1 + k)^{2n - 1} \\
    & \leq
      4n (\sum_{k = 1}^{m-1} (1 + k)^{2n - 1} + 4n (\sum_{k = 1}^{m-1} (1 + m)^{2n-2} + \dotsc )) \\
    & =
      \sum_{l = 1}^{2n} (4n)^l \sum_{k = 1}^{m-1} (1 + k)^{2n - l} \\
    & \leq
      \sum_{l = 1}^{2n} (4n)^l \sum_{k = 1}^{m-1} (1 + k)^{2n - 1} \\
    & \leq
      (4n)^{2n + 1} \sum_{k = 1}^{m-1} (1 + k)^{2n - 1},
  \end{align*}
 where the second inequality uses iteratively a variant of the first inequality.
  Let us now combine these two estimates.  We first infer that 
  \begin{align*}
    \|a\|_{2,n}^2
    & =
      \sum_w |a_w|^2 (1 + |w|)^{2n} \\
    & \leq
      |a_e|^2 + \sum_{w \colon |w| = 1} |a_w|^2 2^{2n} + \sum_{w \colon |w| \geq 2} |a_w|^2 (4n)^{2n + 1} \sum_{k = 1}^{|w| - 1} (1 + k)^{2n - 1}  \\
    & =
      |a_e|^2 + \sum_{w \colon |w| = 1} |a_w|^2 2^{2n} + (4n)^{2n + 1} \sum_{k = 1}^\infty  (1 + k)^{2n - 1}  \sum_{w \colon |w| > k} |a_w|^2 \\
    & \leq
      2^{2n}(|a_e|^2 + \sum_{w \colon |w| = 1} |a_w|^2) + (4n)^{2n + 1} \sum_{k = 1}^\infty  (1 + k)^{2n - 1}  \sum_{w \colon |w| > k} |a_w|^2 \\
    & \leq
      2^{2n} \|a\|^2 + (4n)^{2n + 1} \sum_{k = 1}^\infty  (1 + k)^{2n - 1}  \sum_{w \colon |w| > k} |a_w|^2
      \eqstop
  \end{align*}
  Continuing to estimate the second summand, find that
  \begin{gather*}
    (4n)^{2n + 1} \sum_{k = 1}^\infty  (1 + k)^{2n - 1}  \sum_{w \colon |w|  > k} |a_w|^2
    \leq
    (4n)^{2n + 1} \sum_{k = 1}^\infty  \frac{(1 + k)^{2n - 1}}{k^{2(n+1)}} \nu_{n+1, 1/2}(a)^2
    \eqstop
  \end{gather*}
  So for $C = (\max\{2^n, (4n)^{2n + 1} \sum_{k = 1}^\infty  \frac{(1 + k)^{2n - 1}}{k^{2(n+1)}})^{1/2} \})^{1/2}$, we obtain the inequality
  \begin{gather*}
    \|a\|_{2, n}^2
    \leq
    C^2 (\|a\|^2 + \nu_{n+1, 1/2}(a)^2)
    \leq
    C^2 (\|a\| + \nu_{n+1, 1/2}(a))^2
  \end{gather*}
  and thus the desired estimate $\|a\|_{2,n}  \leq C (\|a\| + \nu_{n+1,1/2}(a))$.
  
  Let us now consider the second item.  Let $a \in \ltwo_{\mathrm{RD}}(W, S, q)$ and write $a = \sum_w a_w T_w$, convergent in $\ltwo(W)$.  We put $b = \sum_w a_w (1 + |w|)^{\frac{r}{\alpha}} T_w$.  Then $b \in \ltwo_{\mathrm{RD}}(W,S,q)$ and hence $b \in \Cstarred(W, S, q)$ by Proposition~\ref{prop:inclusion-rd-algebra}.  By Proposition~\ref{prop:inclusion-rd-algebra}  there is $C_0 > 0$ satisfying $\|c\| \leq C_0 \|c\|_{2,2}$ for all $c \in \Cstarred(W, S, q)$.  We will use this and the fact that for every $u \in \mathrm{Prod}(w,x)$ we have $|u| \leq |w| + |x|$, implying that if $|u| > n$ and $|x| \leq n - n^\alpha$, then $|w| \geq n^\alpha$ holds.  Let us calculate
  \begin{align*}
    \|(1 - P_n)aP_{n - n^\alpha} \xi\|^2
    & \leq
      \sum_{|u| > n} \bigl ( \sum_{w,x \in W: u \in \mathrm{Prod}(w,x)} p_{u,w,x} |a_{w}| (P_{n - n^\alpha}\xi)(x)  \bigr )^2 \\
    & \leq 
      (1 + n^\alpha)^{-\frac{2r}{\alpha}} \sum_{|u| > n} \bigl ( \sum_{w,x \in W:u \in \mathrm{Prod}(w,x)}  p_{u,w,x} |a_{w}| (1 + |w|)^{\frac{r}{\alpha}} (P_{n - n^\alpha}\xi)(x)  \bigr )^2 \\
    & =
      (1 + n^\alpha)^{-\frac{2r}{\alpha}} \|( 1 - P_n) b P_{n - n^\alpha} \xi\|^2 \\
    & \leq
      (1 + n^\alpha)^{-\frac{2r}{\alpha}} C_0^2 \|b\|_{2,2}^2  \|\xi\|^2 \\
    & =
      (1 + n^\alpha)^{-\frac{2r}{\alpha}} C_0^2 \|a\|_{2,2 + \frac{r}{\alpha}}^2 \|\xi\|^2
      \eqstop
  \end{align*}
  Put $C = C_0 \cdot \sup_n \frac{(1 + n^\alpha)^{\frac{r}{\alpha}}}{(1 + n)^r}$.  Then we find that
  \begin{gather*}
    \nu_{r, \alpha}(a) \leq C \|a\|_{2,2 + \frac{r}{\alpha}}
    \eqstop
  \end{gather*}
  So the second item is proven with the constant $C$ and $n = 2 + \frac{r}{\alpha}$.
\end{proof}

The proof of the following lemma is essentially identical to the first part of the proof of \cite[Theorem 1.4]{jolissaint1989-k-theory}, which treats the case of groups with property RD.  For the reader's convenience, we provide full details.
\begin{lemma}
  \label{lem:propagation-norm-power-closed}
  Let $(W, S)$ be a right-angled Coxeter system and $q \in \RR_{> 0}^S$ a deformation parameter.  Let $r \geq 0$, $\alpha \in (0,1)$ and $C < 1$.  There is $s \geq 0$ and a function $h\colon \NN_0 \lra \RR_{\geq 0}$ vanishing at infinity such that for all $a \in \Cstarred(W, S, q)$ satisfying $\|a\| \leq C$ and all $1 \leq M \leq N$, we have
  \begin{gather*}
    \nu_{r, \alpha}(\sum_{l = M}^N a^l) \leq h(M) (\nu_{s, \alpha/2}(a) + \|a\|)
    \eqstop
  \end{gather*}
\end{lemma}
\begin{proof}
  Let $C < 1$ and let $a \in \Cstarred(W, S, q)$ satisfy $\|a\| \leq C$.  For every $n \geq 1$ and every $1 \leq l \leq n^{\alpha/2}$, we have  
  \begin{equation}
          \label{eq:propagation-decomposition}
    \begin{split}
      & \hspace{1em}
        (1 - P_n) a^l P_{n - n^\alpha} \\
      & =
        (1 - P_n) a P_{n - n^{\alpha/2}} a^{l-1} P_{n - n^\alpha}
        +
        (1 - P_n) a (1 - P_{n - n^{\alpha/2}}) a^{l-1} P_{n - n^\alpha} \\
      & =
        (1 - P_n) a P_{n - n^{\alpha/2}} a^{l-1} P_{n - n^\alpha}
        +
        (1 - P_n) a (1 - P_{n - n^{\alpha/2}}) a P_{n - 2n^{\alpha/2}} a^{l-2} P_{n - n^\alpha} \\
      & \qquad
        + \dotsm +
        (1 - P_n) \bigl ( \prod_{i = 1}^{l-1} a (1 - P_{n - i n^{\alpha/2}}) \bigr ) a P_{n - n^\alpha}
        \eqstop
    \end{split}
  \end{equation}
  For $1 \leq i \leq l - 1$ write $m = n - (i -1)n^{\alpha/2}$.  Then we have
  \begin{gather*}
    m - m^{\alpha/2} = n - (i -1)n^{\alpha/2} - (n - (i -1)n^{\alpha/2})^{\alpha/2} \geq n - in^{\alpha/2}
  \end{gather*}
  and hence $P_{m - m^{\alpha/2}} \geq P_{n - in^{\alpha/2}}$.

  So for arbitrary $s \geq 0$, we can consider the norm of the $i$-th summand in \eqref{eq:propagation-decomposition} and find that
    \begin{align*}
    \|(1 - P_n) & a (1 - P_{n - n^{\alpha/2}}) a (1 - P_{n - 2n^{\alpha/2}}) a \dotsm (1 - P_{n - (i -1)n^{\alpha/2}}) a P_{n - i n^{\alpha/2}} a^{l-i} P_{n - n^\alpha}\| \\
    & \leq
      \|(1 - P_{n - (i -1)n^{\alpha/2}}) a P_{n - i n^{\alpha/2}}\| \\
    & =
      \|(1 - P_m) a P_{n - i n^{\alpha/2}}a^* (1 - P_m)\|^{1/2} \\
    & \leq
      \|(1 - P_m) a P_{m - m^{\alpha/2}}a^* (1 - P_m)\|^{1/2} \\
    & =
      \|(1 - P_m) a P_{m - m^{\alpha/2}}\| \\
    & \leq
      m^{-s}\nu_{s, \alpha/2}(a)
      \eqstop
  \end{align*}
  We thus obtain for $1 \leq l \leq n^{\alpha/2}$ that
  \begin{gather*}
    \bigl \| (1 - P_n) a^l P_{n - n^\alpha} \bigr \|
    \leq
    \sum_{i = 0}^{l - 1} (n - in^{\alpha/2})^{-s}\nu_{s, \alpha/2}(a)
    \leq
    \sum_{i = 0}^{l - 1} (n^{\alpha/2})^{-s}\nu_{s, \alpha/2}(a)
    =
    l n^{-\frac{\alpha s}{2}}\nu_{s, \alpha/2}(a)
    \eqstop
  \end{gather*}
  We infer that we also have
  \begin{gather*}
    \bigl \| (P_{n - n^\alpha}) a^l (1 - P_n) \bigr \|
    =
    \bigl \| (1 - P_n) (a^*)^l (P_{n - n^\alpha})  \bigr \|
    \leq
    l n^{-\frac{\alpha s}{2}}\nu_{s, \alpha/2}(a^*)
    =
    l n^{-\frac{\alpha s}{2}}\nu_{s, \alpha/2}(a)
    \eqstop
  \end{gather*}
  Let us now estimate the propagation seminorm of $\sum_{l = M}^N a^l$.  By definition, we have
  \begin{gather*}
    \nu_{r, \alpha} (\sum_{l = M}^N a^l)
    =
    \sup_{n \geq 1} \Bigl ( n^r \bigl (\|(1 - P_n)( \sum_{l = M}^N a^l) (P_{n - n^\alpha})\| + \|(P_{n - n^\alpha}) (\sum_{l = M}^N a^l) (1 - P_n)\| \bigr ) \Bigr ) 
  \end{gather*}
  Since $\|(P_{n - n^\alpha}) (\sum_{l = M}^N a^l) (1 - P_n)\| = \|(1 - P_n) (\sum_{l = M}^N (a^*)^l) (P_{n - n^\alpha})\|$, it suffices to consider the first summand in the supremum.  We find  
  \begin{equation}
    \label{eq:propagation-seminorm-estimate}
    \begin{split}
      & \hspace{1em} \sup_{n \geq 1} n^r \| (1 - P_n) \sum_{l = M}^{N} a^l (P_{n - n^\alpha}) \| \\
      & \leq    
        \sup_{n \geq 1} \left ( n^r \| (1 - P_n)\sum_{l = M}^{\lfloor n^{\alpha/2} \rfloor} a^l P_{n - n^\alpha}\| + n^r \sum_{l = \max(\lceil n^{\alpha/2} \rceil,  M)}^\infty \|a^l\| \right ) \\
      & \leq    
        \sup_{n \geq M} \left ( n^r \| (1 - P_n)\sum_{l = M}^{\lfloor n^{\alpha/2} \rfloor} a^l P_{n - n^\alpha}\| \right ) + \sup_{n \geq 1} \left (n^r\|a\|^{\max{(\lceil n^{\alpha/2} \rceil, M)}} \frac{1}{1 - \|a\|} \right )
        \eqstop
    \end{split}
  \end{equation}
  We estimate both summands.  For the first one we find
  \begin{align*}
    \sup_{n \geq M} \left ( n^r \| (1 - P_n)\sum_{l = M}^{n^{\alpha/2}} a^l P_{n - n^\alpha}\| \right )
    & \leq
      \sup_{n \geq M} \left ( n^r \sum_{l = M}^{n^{\alpha/2}} l n^{-\frac{\alpha s}{2}}\nu_{s, \alpha/2}(a) \right ) \\
    & \leq
      \sup_{n \geq M} \left ( n^r \frac{n^\alpha + n^{\alpha/2}}{2}  n^{-\frac{\alpha s}{2}} \right ) \nu_{s, \alpha/2}(a) \\
    & \leq
      \sup_{n \geq M} \left ( n^{r + \alpha - \frac{\alpha s}{2}} \right ) \nu_{s, \alpha/2}(a)
      \eqstop
  \end{align*}
  For $s \geq \frac{2(r + \alpha)}{\alpha}$ we can estimate the last term from above by $M^{r + \alpha - \frac{\alpha s}{2}} \nu_{s , \alpha/2}(a)$.

  The second summand in \eqref{eq:propagation-seminorm-estimate} can be estimated as follows.  For $n^{\alpha/2} \leq M$, we have
  \begin{gather*}
    n^r \|a\|^{\max{(\lceil n^{\alpha/2} \rceil, M)}} \frac{1}{1 - \|a\|}
    =
    (n^{\alpha/2})^{\frac{2r}{\alpha}} \|a\|^{\max{(\lceil n^{\alpha/2} \rceil, M)}} \frac{1}{1 - \|a\|}
    \leq
    M^{\frac{2r}{\alpha}} \frac{1}{1 - C} \|a\|^M 
    \eqstop
  \end{gather*}
  Further, we have
  \begin{align*}
    \sup_{n \geq M^{2/\alpha}} \left ( n^r\|a\|^{\max{(\lceil n^{\alpha/2} \rceil, M)}} \frac{1}{1 - \|a\|} \right )
    & \leq
    \sup_{n \geq M^{2/\alpha}} \left ( n^rC^{\lceil n^{\alpha/2} \rceil - M}  \right ) \|a\|^M \frac{1}{1 - C}
  \end{align*}
  We infer that there is constant $c_{\alpha, r} > 0$ and an integer $d \geq 1$ such that
  \begin{gather*}
    \sup_{n \geq 1} \left ( n^r\|a\|^{\max{(\lceil n^{\alpha/2} \rceil, M)}} \frac{1}{1 - \|a\|} \right )
    \leq
    c_{\alpha, r} M^d \|a\|^M
    \eqstop
  \end{gather*}
  Observing that the analogous estimates hold for $a^*$, we can continue the estimate from \eqref{eq:propagation-seminorm-estimate} and find that for $s > \frac{2(r + \alpha)}{\alpha}$, there is a function $h\colon \NN_0 \lra \RR_{\geq 0}$ vanishing at infinity such that, 
  \begin{gather*}
    \nu_{r, \alpha}(\sum_{l = M}^N a^l)
    \leq
    2 \left (M^{r + \alpha - \frac{\alpha s}{2}} \nu_{s , \alpha/2}(a) + C_{\alpha, r} M^d \|a\|^M \right )
    \leq h(M) (\nu_{s, \alpha/2}(a) + \|a\|)
  \end{gather*}
  holds for all $1 \leq M \leq N$.
\end{proof}

The next result combines the work of Caspers-Klisse-Larsen \cite{caspersklisselarsen2021} reformulated in Proposition \ref{prop:inclusion-rd-algebra} and the estimates on propagation seminorms with the notion of empowered differential seminorms as introduced in Section~\ref{sec:smooth-subalgebras-revisited}.  As a result we show that the Sobolev norms give rise to an unconditional Fr{\'e}chet completion of the Hecke algebra which is a smooth subalgebra of the reduced Hecke-C*-algebra.
\begin{theorem}
  \label{thm:ltwo-rd-smooth}
  Let $(W, S)$ be a right-angled Coxeter system and let $q \in \RR_{> 0}^S$ be a deformation parameter.  Then the Sobolev norms $\rD_n(a) = \|a\|_{2,n}$ define an empowered differential seminorm $\rD$ on $\CC(W, S, q)$.
\end{theorem}
\begin{proof}
 We put $\rD'_n(a) = \nu_{n, 1/2}(a)$ for $a \in \CC(W, S, q)$.  As Sobolev norms satisfy $\|a\|_{2,s} \leq \rD_{\lceil s \rceil}(a)$ for all $a \in \CC(W, S, q)$ and $s \geq 0$, Proposition~\ref{prop:empowered-differential-seminorm-equivalence} allows us to combine the estimates for  $\rD_n(ab)$ with $a, b \in \CC(W, S, q)$ proven in Propostion~\ref{prop:inclusion-rd-algebra} with the estimates on power series obtained for $\rD'$ in Lemma~\ref{lem:propagation-norm-power-closed} to infer that $\rD$ is an empowered differential seminorm on $\CC(W, S, q)$.
\end{proof}

\begin{corollary}
  \label{cor:unconditional-differential-seminorm}
  The algebra $\ltwo_{\mathrm{RD}}(W, S, q) \subseteq \Cstarred(W, S, q)$ is smooth.  It is the completion of $\CC(W, S, q)$ with respect to an unconditional empowered differential seminorm.
\end{corollary}
\begin{proof}
  By Theorem~\ref{thm:ltwo-rd-smooth} the Sobolev norms define an empowered differential seminorm on $\CC(W, S, q)$.  The relevant completion is $\ltwo_{\mathrm{RD}}(W, S, q)$ by Remark~\ref{rem:sobolev-norm-monotone}.  Its components are esaily seen to be unconditional.
\end{proof}

\begin{theorem}
  \label{thm:quasi-derivation-hecke-case}
  Let $(W, S)$ be a right-angled, hyperbolic Coxeter system and $q \in \RR_{>0}^{S}$ a deformation parameter.  Then there is a constant $C > 0$ such that the assignment
  \begin{gather*}
    \Delta \colon T_w \mapsto \sum_{\substack{w = w_1 w_2 \\ |w| = |w_1| + |w_2|}} T_{w_1} \otimes T_{w_2}
  \end{gather*}
  defines a quasi-derivation with constant $C$ for every unconditional norm on $\CC(W, S, q)$.
\end{theorem}
\begin{proof}
  Fix an unconditional norm on $\CC(W, S, q)$, which we denote by $\| \cdot \|$.  The associated unconditional tensor product norm on $\CC(W, S, q) \otimes \CC(W, S, q)$ will likewise be denoted by $\| \cdot \|$.  We will show that there is a bounded map $\vphi$ on $\CC(W, S, q) \ot \CC(W, S, q)$ satisfying
  \begin{gather*}
    \Delta(T_wT_x) \leq \vphi( (T_w \ot 1) \Delta(T_x) + \Delta(T_w) (T_x \ot 1))
  \end{gather*}
  for all $w,x \in W$.  Then linearity of both sides of the inequality shows that for $\sum_i a_i T_{w_i}, \sum_i b_i T_{x_i} \in \CC(W, S, q)$ we have
  \begin{align*}
    \Delta \bigl ((\sum_i a_iT_{w_i}) (\sum_j b_j T_{x_j}) \bigr )
    & =
      \sum_{i,j} a_ib_j \Delta(T_{w_i} T_{x_j}) \\
    & \leq
      \sum_{i,j} a_ib_j  \vphi \bigl ( (T_{w_i} \ot 1) \Delta(T_{x_j}) + \Delta(T_{w_i})( 1 \ot T_{x_j}) \bigr ) \\
    & =
      \vphi \bigl ( (\sum_i a_i T_{w_i} \ot 1)  \Delta(\sum_j b_j T_{x_j}) + \Delta(\sum_i a_i T_{w_i}) (1 \ot \sum_j b_jT_{x_j}) \bigr )
      \eqstop
  \end{align*}
  As the tensor product norm is an unconditional cross-norm, the statement of the theorem follows with $C = \|\vphi\|$.

  For $r > 0$ denote by $\vphi_r \colon \CC(W, S, q) \ot \CC(W, S, q) \lra \CC(W, S, q) \ot \CC(W, S, q)$ the linear map such that for all $w, x \in W$ we have
  \begin{gather*}
    \vphi_r(T_w \otimes T_x)
    =
    \sum_{a,b \in \rB_r(e)} T_w T_a \otimes T_b T_x
    \eqstop
  \end{gather*}
  As left and right multiplication by generators $T_s, s \in S$ is bounded, there is a constant $D > 0$ such that $\|\vphi_r\| \leq |\rB_r(e)|^2 D$ for all $r > 0$.

  Let $w,x \in W$ be fixed elements and write
  \begin{gather*}
    T_wT_x = \sum_{u \in \mathrm{Prod(w,x)}} p_{u, w, x} T_u
    \eqstop
  \end{gather*}
  Let $P \subseteq S^*$ be a set of reduced words which is a complete set of representatives for $\mathrm{Prod}(w,x)$.  For all $u \in \mathrm{Prod}(w,x)$ let $s_1 \dotsm s_k \in P$ be the unique representative.  As $W$ is assumed hyperbolic, there is $D' > 0$ such that any geodesic from $e$ to $u$ lies at distance at most $D'$ from $s_1 \dotsm s_k$.  Hence
  \begin{gather*}
    \Delta(T_u) \leq \vphi_{D'} (\sum_{i = 0}^k T_{s_1\dotsm s_i} \otimes T_{s_{i+1} \dotsm s_k})
    \eqstop
  \end{gather*}
  It follows that
  \begin{gather*}
    \Delta(T_wT_x)
    =
    \sum_{u \in \mathrm{Prod(w,x)}} p_{u, w, x} \Delta(T_u)
    \leq
    \sum_{u \in P}
    p_{u, w, x}
    \vphi_{D'} (\sum_{\substack{u = u_1 u_2 \\ \text{ reduced decomposition}}} T_{u_1} \otimes T_{u_2})
    \eqstop
  \end{gather*}
  By Theorem~\ref{thm:skipping-cancellation-normal-form} we may choose $P$ such that every word in there can be written in the form $w_1 t_1 \dotsm t_k x_2$ for reduced decompositions $w = w_1 t_1 \dotsc t_k x_1^{-1}$ and $x = x_1 t_1 \dotsc t_k x_2$ such that the letters $t_1, \dotsc, t_k$ pairwise commute. Let $u \in P$ and consider corresponding reduced decompositions of $w$ and $x$ as above.  Let $u = u_1 u_2$ be any reduced decomposition.  If $u_1$ is an initial piece of $w_1$, denote by $w_2$ the subword of $w = w_1 t_1 \dotsc t_k x_1^{-1}$ satisfying $w = u_1w_2$.  Since $u = u_1 u_2$ and $w = u_1 w_2$ are reduced decompositions, it follows that $p_{u,w,x} = p_{u_2,w_2,x}$.  So we find
  \begin{gather*}
    p_{u,w,x} T_{u_1} \otimes T_{u_2}
    \leq
    T_{u_1} \otimes T_{w_2}(1 \ot T_x)
    \leq
    \Delta(T_w)(1 \ot T_x)
    \eqstop
  \end{gather*}
  Similarly, if $u_1$ contains $w_1 t_1 \dotsc t_k$ as an initial piece, we find that
  \begin{gather*}
    p_{u,w,x} T_{u_1} \otimes T_{u_2}
    \leq
    (T_w \ot 1)\Delta(T_x)
    \eqstop
  \end{gather*}
  Let us consider the case $u_1 = w_1 t_1 \dotsm t_i$ and $u_2 = t_{i+1} \dotsm t_k x_2$ for some $1 \leq i < k$.  Then
  \begin{align*}
    p_{u,w,x} T_{u_1} \otimes T_{u_2}
    & \leq
    p_{u,w,x} \vphi_k(T_{w_1} \ot T_{t_1 \dotsm t_k x_2}) \\
    & \leq
    \vphi_k((T_{w_1} \ot T_{t_1 \dotsm t_k x_1^{-1}})(1 \ot T_x)) \\
    & \leq
    \vphi_k(\Delta(T_w)(1 \ot T_x))
    \eqstop
  \end{align*}
  Summarising, we find that
  \begin{align*}
    \Delta(T_wT_x)
    & \leq
      \sum_{u \in P} p_{u,w,x} \vphi_{D'} \bigl (\sum_{u = u_1 u_2 \text{ reduced decomposition}} T_{u_1} \otimes T_{u_2} \bigr ) \\
    & \leq
      \vphi_{D'} \circ \vphi_k (\Delta(T_w)(1 \ot T_x) + (T_w \ot 1)(\Delta(T_x)))
      \eqstop
  \end{align*}
  As $k \leq |S|$, we can conclude the proof.
\end{proof}

\begin{corollary}
  \label{cor:hecke-schwartz-algebra-continuous-quasi-derivation}
  Let $(W, S)$ be a right-angled, hyperbolic Coxeter system and let $q \in \RR_{>0}^S$ be a deformation parameter.  Then there are Fr{\'e}chet algebra completions $\cS(W, S, q)$ and $\cS'(W, S, q)$ of $\CC(W, S, q)$ such that
  \begin{enumerate}
  \item the inclusion $\CC(W, S, q) \subseteq \Cstarred(W, S, q)$ continuously extends to an inclusion $\cS'(W, S, q) \subseteq \Cstarred(W, S, q)$, identifying $\cS'(W, S, q)$ as a smooth subalgebra of $\Cstarred(W, S, q)$;
  \item the completion $\CC(W, S, q) \subseteq \cS(W, S, q)$ is unconditional and can be identified with a subalgebra of $\cS'(W, S, q)$, and
  \item the natural quasi-derivation $\Delta\colon \CC(W, S, q) \lra \CC(W, S, q) \ot \CC(W, S, q)$ continuously extends to a map $\cS'(W, S, q) \lra \cS(W, S, q) \ot \cS(W, S,q)$, where the tensor product is the unconditional one.
  \end{enumerate}
 We thus obtain a chain of continuous inclusions of Fr\'echet algebras
 \begin{gather*}
   \cS'(W, S, q) \subseteq \cS(W, S, q) \subseteq \Cstarred(W, S, q),
 \end{gather*}
 where both $\cS'(W, S, q)$ and $\cS(W, S, q)$ are smooth. 
\end{corollary}
\begin{proof}
  By Corollary~\ref{cor:unconditional-differential-seminorm} there is an unconditional empowered differential seminorm $\rD$ on $\CC(W, S, q)$ such that the inclusion $\CC(W, S, q) \hra \Cstarred(W, S, q)$ continuously extends to an inclusion of the completion $\cS(W, S, q)$ with respect to $\rD$.  By Theorem~\ref{thm:smooth-subalgebra-differential-seminorm}, we infer that $\cS(W, S, q) \subseteq \Cstarred(W, S, q)$ is smooth.

  By Theorem~\ref{thm:quasi-derivation-hecke-case} there is a constant $C > 0$ such that the natural quasi-derivation on $\CC(W, S, q)$ is a quasi-derivation with constant $C$ for every component of $\rD$.  Hence, by Proposition \ref{prop:quasi-derivation-gives-differential-seminorm} there is an empowered differential seminorm $\rD'$ defined by 
  \begin{gather*}
    \rD'_0(a) = C \rD_0(a)
    \eqcomma \qquad
    \rD_n'(a) = \rD_n(a) + \rD^\cM_{n-1}(\delta(a))
    \quad \text{for all } a \in \CC(W, S, q), n \geq 1
    \eqstop
  \end{gather*}
  Denoting by $\cS'(W, S, q)$ the corresponding completion, we obtain a continuous extension of the identity map to an inclusion $\cS'(W, S, q) \lra \cS(W, S, q)$.  We hence have a continuous inclusion $\cS'(W, S ,q) \subseteq \Cstarred(W, S, q)$.  Since $\cS'(W, S, q)$ is the completion with respect to an unconditional empowered differential seminorm, it becomes a smooth subalgebra of $\Cstarred(W, S, q)$ by Theorem~\ref{thm:smooth-subalgebra-differential-seminorm}.

  By definition of $\rD'$ the natural quasi-derivation continuously extends to a map $\cS'(W, S, q) \lra \cS(W, S, q) \ot \cS(W, S, q)$.
\end{proof}

\section{Traces on right-angled Hecke algebras}
\label{sec:traces-ra-hecke-algebras}

The aim of this section is to describe the space of tracial functionals on a right-angled Hecke algebra.  To this end, we recall the following approach, employed by He--Nie in \cite{henie2014}.  The vector space of tracial functionals on an algebra $\cA$ is isomorphic to the dual space of the cocentre $\cA / [\cA, \cA]$, where $[\cA, \cA]$ denotes the subspace generated by additive commutators $ab - ba$ for $a,b \in \cA$. It is shown in \cite{henie2014} that for an extended affine Hecke algebra $\cH$ associated with extended Weyl group $\tilde W$ and defined over $\ZZ[v,v^{-1}]$, any two elements of minimal length from a conjugacy class in $\tilde W$ define the same element in the cocentre and that the elements so obtained form a basis.  In the more general setup of generic Hecke algebras of arbitrary Coxeter groups, we will use Marquis' work to adapt this result.

The proof of the following proposition is a direct adaption of the one of \cite[Lemma 5.1]{henie2014}.
\begin{proposition}
  \label{prop:conjucacy-classes-to-cocentre}
  Let $\cH$ be a generic Hecke algebra of type $(W, S)$.  Let $w,w' \in W$ be elements of minimal length in the same conjugacy class of $W$.  Then $T_w + [\cH, \cH] = T_{w'} + [\cH, \cH]$.
\end{proposition}
\begin{proof}
  Let $w, w'$ be two different elements of minimal length in the same conjugacy class of $W$.  By \cite{marquis2020-structure-conjugacy-class} there is a sequence of elementary tight conjugations that conjugate $w'$ to $w$.  It hence suffices to prove the statement for the case where $w$ is obtained from $w'$ by elementary tight conjugation (for the meaning of the latter see \cite{marquis2020-structure-conjugacy-class}).  So assume that there is an element $x \in W$ satisfying $x^{-1}w'x = w$ and $|w'x| = |w'| + |x|$ or $|x^{-1}w'| = |x| + |w'|$.  By symmetry, we may assume the former.  Then $|xw| = |xx^{-1}w'x| = |w| + |x|$ follows.

  Calculating modulo the commutator subspace, we find that
  \begin{gather*}
    T_w
    =
    T_x^{-1} T_x T_w
    =
    T_x^{-1} T_{xw}
    \equiv
    T_{xw} T_x^{-1}
    =
    T_{w'x}T_x^{-1}
    =
    T_{w'}T_xT_x^{-1}
    =
    T_{w'} \pmod{[\cH,\cH]}
    \eqstop
  \end{gather*}
  This finishes the proof.
\end{proof}

\begin{notation}
  \label{not:conjugacy-classes-cocentre}
  Let $\cH$ be a generic Hecke algebra of type $(W,S)$.  For a conjugacy class $\cO \subseteq W$ we denote by $T_\cO$ the unique element of the cocentre of $\cH$ obtained as the image of $T_w$ for a minimal length element in $\cO$.
\end{notation}

The proof of \cite[Theorem 5.3]{henie2014} applies verbatim to show the following result: the collection $(T_\cO)_{\cO}$ linearly generates the cocentre of any generic Hecke algebra.  We will use the following description of the coefficients appearing in linear combinations of these elements.
\begin{lemma}
  \label{lem:non-minimal-element-cocentre}
  Let $\cH$ be a generic Hecke algebra of type $(W,S)$.  Let $\cO \subseteq W$ be a conjugacy class and let $w \in \cO$.  Then
  \begin{gather*}
    T_w + [\cH, \cH]
    =
    \sum_\cK c_{\cK, w} T_\cK
  \end{gather*}
  for certain coefficients $c_{\cK, w}$, where $\cK$ runs over all conjugacy classes of elements that are Bruhat subordinate to $w$.
\end{lemma}
\begin{proof}
  Let $\ell$ be the minimal length of an element in $\cO$.  We prove the claim by induction on $|w| - \ell = k \in \NN_0$.  For $k = 0$ there is nothing to prove.

  If $|w| - \ell \geq 1$ and the statement is proven for all $k < |w| - \ell$, we allude to Marquis' Theorem~\ref{thm:cyclic-reduction-shifts-right-angled} and infer that $w$ is not cyclically reduced.  So there is $s \in S$ such that for $w' = sws$ we have $|w'| = |w| - 2$.  Then also $|w's| > |w'|$ and hence denoting by $(a_s)_s$ and $(b_s)_s$ the deformation parameters of $\cH$, we find that
  \begin{gather*}
    T_w
    =
    T_s T_{w'}T_s
    \equiv
    T_{w'}T_s^2
    =
    T_{w'}(a_sT_s + b_s)
    =
    a_s T_{w's} + b_sT_{w'}
    \pmod{[\cH,\cH]}
    \eqstop
  \end{gather*}
  So the induction hypothesis applies to $T_{w's}$ and $T_{w'}$.  As both $w'$ and $w's$ are Bruhat subordinate to $w$, this finishes the proof.
\end{proof}

For the following lemma we have to specialise to the case of right-angled Coxeter systems.
\begin{lemma}
  \label{lem:bruhat-order-conjugacy-classes}
  Let $(W,S)$ be a right-angled Coxeter system and define a relation on conjugacy classes of $W$ by
  \begin{gather*}
    \cO \leq \cK
    \quad \text{ if and only if } \quad
    \exists \, g \in \cO \,\exists \, \text{minimal length element } k \in \cK\colon g \text{ is Bruhat subordinate to } k
    \eqstop
  \end{gather*}
  Then $\leq$ is a partial order.
\end{lemma}
\begin{proof}
  It is clear that $\leq$ is reflexive.  We have to show that it is antisymmetric and transitive.  Assume that $\cO \leq \cK$ and $\cK \leq \cO$.  Let $g_1, g_2 \in \cO$, $k_1, k_2 \in \cK$ be such that $g_2$ and $k_1$ are of minimal length in $\cO$ and $\cK$, respectively, and the Bruhat subordination relations $g_1 \leq k_1$ and $k_2 \leq g_2$ hold.  Then
  \begin{gather*}
    |g_1| \leq |k_1| \leq |k_2| \leq |g_2| \leq |g_1|
  \end{gather*}
  showing that $|g_1| = |k_1|$.  Since $g_1$ is Bruhat subordinate to $k_1$, we conclude that $g_1 = k_1$ and hence $\cO = \cK$ holds.

  Assume now that $\cO \leq \cK \leq \cL$.  Let $g \in \cO$, $k,k' \in \cK$, $l \in \cL$ such that $k$ and $l$ have minimal length in $\cK$ and $\cL$, respectively, and $g \leq k$ as well as $k' \leq l$ holds.  Since $(W, S)$ is right-angled, Marquis' Theorem~\ref{thm:cyclic-reduction-shifts-right-angled} says that there exists $n \in \NN_0$ and   sequences of elements $ k_1, \dotsc, k_n$ and $s_1, \dotsc, s_{n-1}$ such that $k_n = k'$, $k_i = s_i k_{i+1}s_i$ with $|k_i| < |k_{i+1}|$ for all $i \in \{1, \dots, n-1\}$ and $k$ is a cyclic shift of $k_1$.  In particular, $k_1 \leq k'$.  Since $k$ is a cyclic shift of $k_1$ and $g \leq k$, there is a cyclic shift $g_1$ of $g$ such that $g_1 \leq k_1$.  It follows now that $g_1 \leq k_1 \leq k' \leq l$.  Since $g_1 \in \cO$, this shows $\cO \leq \cL$ and finishes the proof.
\end{proof}

\begin{remark}
  \label{rem:bruhat-order-conjugacy-classes-minimal-length}
  We remark that every element in a Coxeter system $(W, S)$ Bruhat dominates a minimal length element from its conjugacy class, since it can be cyclically reduced to such by Marquis' result \cite[Theorem A]{marquis2021-cyclically-reduced}.  As a consequences, in the context of Lemma~\ref{lem:bruhat-order-conjugacy-classes} we also have that
  \begin{gather*}
    \cO \leq \cK
    \quad \text{ if and only if } \quad
    \exists g \in \cO, k \in \cK \text{minimal length elements} \colon g \text{ Bruhat subordinate to } k
    \eqstop
  \end{gather*}
\end{remark}

\begin{theorem}
  \label{thm:right-angled-cocentre-basis}
  Let $\cH$ be a generic Hecke algebra of right-angled type $(W,S)$ over an algebraically closed field $K$ of characteristic different from $2$. Let $(a_s)_{s\in S}$, $(b_s)_{s\in S}$ be the deformation parameters of $\cH$ and assume that $b_s \neq -\frac{a_s^2}{4}$ for each $s \in S$.  Then $(T_{\cO})_{\cO}$ where $\cO$ runs through the conjugacy classes of $W$ is a basis of the cocentre of $\cH$.
\end{theorem}
\begin{proof}
  Using the assumptions on $K$, Proposition~\ref{prop:isomorphism-ra-hecke-algebra} applies to show that there is an isomorphism of unital $K$-algebras $\vphi\colon KW \lra \cH$ satisfying $\vphi(u_s) = c_sT_s + d_s$ for certain $c_s \in K^\times$ and $d_s \in K$.  Then $\vphi$ induces an isomorphism of cocentres $\psi\colon KW/[KW,KW] \lra \cH/[\cH, \cH]$.  For a conjugacy class $\cO$ in $W$ and an element of minimal length $w \in \cO$, we denote by $u_{\cO}$ the image of $u_w$ in the cocentre of $KW$.  Using Proposition~\ref{prop:isomorphism-ra-hecke-algebra}, we then find
  \begin{gather*}
    \psi(u_{\cO})
    =
    \psi(u_w + [KW,KW])
    =
    \vphi(u_w) + [\cH,\cH]
    =
    \sum_{w' \leq w} c_{w',w} T_w' + [\cH, \cH]
  \end{gather*}
  for certain coefficients $c_{w',w} \in K$ with $c_{w,w} = c_{s_1} \dotsm c_{s_n} \in K^\times$ for any reduced word $s_1\dotsm s_n$ representing $w$.  By Lemma~\ref{lem:non-minimal-element-cocentre}, we can now write
  \begin{gather*}
    \psi(u_\cO)
    =
    \sum_{\cK \leq \cO} c_{\cK,\cO} T_\cK
  \end{gather*}
  for certain coefficients $c_{\cK,\cO}$ with $c_{\cO, \cO} = c_{w,w} \in K^\times$.  Since the family $(\psi(u_{\cO}))_{\cO}$ is linearly independent, by Lemma~\ref{lem:order-preserving-base-change} also $(T_\cO)_\cO$ is linearly independent.  Combined with Lemma~\ref{lem:non-minimal-element-cocentre} we conclude that $(T_\cO)_\cO$ is a basis of the cocentre of $\cH$.
\end{proof}

We will now give a description of the dual functionals associated with the basis $(T_\cO)_\cO$ of the cocentre of $\cH$.  To be more precise, we consider the associated traces on $\cH$.  The description given here will be suitable for our purpose to extend these traces to Schwartz algebra completions of Hecke algebras considered in Section~\ref{sec:smooth-subalgebras-hecke}.

\begin{notation}
  \label{not:cyclic-reduction-map}
  Fix a total order $\leq_S$ on $S$ and define a linear map $\Phi_0\colon \cH \lra \cH$ by its evaluation on $T_w$, $w \in W$ being
  \begin{gather*}
    \Phi_0(T_w)
    =
    \begin{cases}
      T_w & \text{ if $w$ is of minimal length in its conjugacy class} \\
      T_{sws}T_s^2 & \text{ if $s \in S$ is $\leq_S$-minimal with $|sws| < |w|$}\eqstop
    \end{cases}
  \end{gather*}
  Define further $\Phi(T_w) = \lim_n \Phi_0^{\circ n}(T_w)$ for $w \in W$.
\end{notation}

\begin{remark}
  \label{rem:cyclic-reduction-map-basics}
  We observe that $\Phi$ is well-defined, since for every $w \in W$ either $\Phi_0(T_w) = T_w$ holds or the two elements in the support of $\Phi_0(T_w)$ have length $|w| - 1$ and $|w| - 2$.  Hence, $\Phi^{\circ k}(T_w) = \Phi^{\circ |w|}(T_w)$ for all $w \in W$ and all $k \geq |w|$.  By Marquis' Theorem~\ref{thm:cyclic-reduction-shifts-right-angled}, every element in $W$ is either minimal in its conjugacy class or can be cyclically reduced.  Hence
  \begin{gather*}
    \im \Phi = \lspan \{ T_w \mid w \text{ has minimal length in its conjugacy class}\}
    \eqstop
  \end{gather*}
\end{remark}

\begin{lemma}
  \label{lem:cyclic-reduction-order-irrelevant}
  Let $(W,S)$ be a right-angled Coxeter system and $\Phi$ be the map on a generic Hecke algebra of type $(W,S)$ defined in Notation~\ref{not:cyclic-reduction-map}.  Let $w \in W$ and assume that $w = s w' s$ is a reduced expression for some $w' \in W$, $s \in S$.  Then $\Phi(T_w) = \Phi(T_{w'} T_s^2)$.
\end{lemma}
\begin{proof}
  We prove the statement by induction on the length of $w$, which must be at least three.  If $|w| = 3$, then $w = sts$ for some $w' = t \in S$ not commuting with $s$.  Let $s' \in S$ be $\leq_S$-minimal with the property that $|s'ws'| \leq |w|$.  Then there is $t' \in S$ such that $w = s't's'$ and it follows that $t' = s' s t s s'$.  As $(W, S)$ is right-angled, this implies $t' = t$ and a fortiori $s' = s$.  So we obtain
  \begin{gather*}
    \Phi_0(T_w)
    =
    T_t T_s^2
    =
    a_s T_t + b_s T_{ts}
    \eqstop
  \end{gather*}
  Since $t$ and $ts$ are minimal length elements in their conjugacy classes, it follows that $\Phi(T_w) = a_s T_t + b_s T_{ts} = \Phi(T_{w'} T_s^2)$.
 
  Assume now that $|w| > 3$.  By assumption $w$ cannot be minimal in its conjugacy class, so there is $t \in S$ such that $w = tw''t$ is reduced and $\Phi_0(T_w) = T_{w''}T_t^2$.   If $t = s$, we finish the proof as above.  Otherwise, we observe that since $w$ ends both with $t$ and $s$, we have $[t,s] = e$.  It follows that $w''$ and $w''t$ end with $s$.  So inductively, we find
  \begin{gather*}
    \Phi(T_{w''}) = \Phi(T_{w'''}T_s^2)
    \quad \text{ and } \quad
    \Phi(T_{w''t}) = \Phi(T_{w'''t}T_s^2)
  \end{gather*}
  for $w''' = sw''s$.  This allows us to conclude by the calculation
  \begin{align*}
    \Phi(T_w)
    & =
      \Phi(T_{w''} T_t^2) \\
    & =
      \Phi(a_t T_{w''} + b_t T_{w''t}) \\
    & =
      \Phi(a_t T_{w'''}T_s^2 + b_t T_{w'''t}T_s^2) \\
    & =
      \Phi(T_{w'''}T_s^2T_t^2) \\
    & =
      \Phi(T_{w'}T_s^2)
      \eqcomma
  \end{align*}
  where the induction hypothesis justifies the third and last equalities.  
\end{proof}

\begin{notation}
  \label{not:delocalised-trace}
  Let $(W, S)$ be a Coxeter system of right-angled type and let $\cH$ be a generic Hecke algebra of type $(W, S)$ defined over $R$.  For a conjugacy class $\cO \subseteq W$, we denote by $\Sigma_\cO\colon \cH \lra R$ the sum over all coefficients of minimal length elements in $\cO$, that is
  \begin{gather*}
    \Sigma_\cO(T_w)
    =
    \begin{cases}
      1 & \text{ if $w \in \cO$ and $w$ is of minimal length in $\cO$} \\
      0 & \text{ otherwise}\eqstop
    \end{cases}
  \end{gather*}
  Combining this with Notation~\ref{not:cyclic-reduction-map}, we will write $\vphi_\cO = \Sigma_\cO \circ \Phi$.
\end{notation}

\begin{theorem}
  \label{thm:delocalised-trace-identification}
  Let $\cH$ be a generic Hecke algebra of right-angled type $(W,S)$.  For every conjugacy class $\cO \subseteq W$, the map $\vphi_\cO$ is a trace on $\cH$.  If $(T_{\cU})_{\cU}$ denotes the basis of the cocentre of $\cH$ satisfying $T_{\cU} = T_w + [\cH, \cH]$ for every element of minimal length $w \in \cU$, then $\vphi_\cO(T_\cU) = \delta_{\cO, \cU}$. Moreover, the assignment $(c_\cO)_{\cO} \mapsto \sum_\cO c_\cO \vphi_\cO$ defines an isomorphism between the space of complex sequences on the set of conjugacy classes of $W$ and the space of tracial functionals on $\cH$.  
\end{theorem}
\begin{proof}
  In order to show that $\vphi_\cO$ is a trace, it suffices to show that for all $w \in W$ and all $s \in S$ the equality $\vphi_\cO(T_s T_w) = \vphi_\cO(T_w T_s)$ holds.  We prove this statement by induction on $|w| \in \NN_0$ and observe that the case $|w| = 0$ is clear.  We distinguish several cases.

  If $|sw| = |ws| > |w|$, then we have $T_sT_w = T_{sw}$ and $T_wT_s = T_{ws}$.  Observe that $sw$ is of minimal length in its conjugacy class if and only if so is $ws$.  In this case we find that
  \begin{gather*}
    \vphi_\cO(T_s T_w)
    =
    \Sigma_\cO (T_{sw})
    =
    \Sigma_\cO (T_{ws})
    =
    \vphi_\cO(T_wT_s)
    \eqstop
  \end{gather*}
  If $sw$ is not of minimal length in its conjugacy class, let $t \in S$ such that $|tswt| < |sw|$.  It follows that $[s,t] = e$ and hence also $|twst| < |ws|$ and $|twt| < |w|$.  Writing $w' = twt$, by Lemma~\ref{lem:cyclic-reduction-order-irrelevant} applied to the reduced expression $t(sw')t$ we infer from the induction hypothesis that
  \begin{gather*}
    \vphi_\cO(T_sT_w)
    =
    \vphi_\cO(T_{sw'} T_t^2)
    =
    \vphi_\cO(a_t T_{sw'} + b_t T_{sw't})
    =
    \vphi_\cO(a_t T_{w's} + b_t T_{w'ts})
    =
    \vphi_\cO(T_wT_s)
    \eqstop
  \end{gather*}

  In case $|sw| > |w| > |ws|$, we write $w = w's$ with $|w| = |w'| + 1$.  Then $w'$ does neither start nor end with $s$, so that
  \begin{gather*}
    T_sT_w = T_{sw's}
    \quad \text{ and } \quad
    T_w T_s = T_{w'}T_s^2
    \eqstop
  \end{gather*}
  Lemma~\ref{lem:cyclic-reduction-order-irrelevant} applies and shows that
  \begin{gather*}
    \vphi_\cO(T_s T_w)
    =
    \Sigma_\cO \circ \Phi(T_{sw's})
    =
    \Sigma_\cO \circ \Phi(T_{w'}T_s^2)
    =
    \vphi_\cO(T_wT_s)
    \eqstop
  \end{gather*}

  In case $|ws| > |w| > |sw|$, we argue similarly.

  We consider the final case $|sw| = |ws| < |w|$. Write $w = sw's$ with $|w| = |w'| + 2$.  We have
  \begin{gather*}
    T_w T_s = a_s T_{sw'} + b_sT_w
    \quad \text{ and } \quad
    T_s T_w = a_s T_{w's} + b_sT_w
    \eqstop
  \end{gather*}
  Since $|sw'| = |w's| > |w'|$, the first case applies to show that
  \begin{gather*}
    \vphi_\cO(T_sT_w)
    =
    \vphi_\cO(a_S T_{w's} + b_sT_w)
    =
    \vphi_\cO(a_s T_{sw'} + b_sT_w)
    =
    \vphi_\cO(T_wT_s)
    \eqstop
  \end{gather*}

  We have now shown that $\vphi_\cO$ is a trace and hence factors through the cocentre of $\cH$.  Let $\cU$ be a conjugacy class in $W$ and let $w \in \cU$ be an element of minimal length.  Then
  \begin{gather*}
    \vphi_\cO(T_\cU) = \Sigma_\cO \circ \Phi(T_w)  = \Sigma_\cO (T_w) = \delta_{\cO, \cU}
    \eqstop
  \end{gather*}
  
  The last statement follows from Theorem~\ref{thm:right-angled-cocentre-basis} and the identification of the dual space $(\cc(I))^* \cong \linfty(I)$ for any set $I$.
\end{proof}

Let us finish this section by providing a proof of the simplified version of Theorem~\ref{thm:delocalised-trace-identification} stated in the introduction.
\begin{proof}[Proof of Theorem~\ref{thmintro:delocalised-trace-identification-simplified}]
Existence of the trace $\vphi_\cO$ is an immediate consequence of Theorem~\ref{thm:delocalised-trace-identification}. Uniqueness follows from the fact that $(T_\cU)_\cU$ is a basis of the cocentre of $\cH$ by Theorem~\ref{thm:right-angled-cocentre-basis}. The second statement of the theorem is part of Theorem~\ref{thm:delocalised-trace-identification}.
\end{proof}

\section{Traces on Hecke-Schwartz algebras of right-angled, hyperbolic type and K-theory pairing}
\label{sec:traces-schwartz}

We are now ready to prove that delocalised traces extend to Hecke-Schwartz algebras, combining results about traces on right-angled Hecke algebras from Section~\ref{sec:traces-ra-hecke-algebras} and our study of Schwartz algebra completions of Iwahori-Hecke algebras in Section~\ref{sec:smooth-subalgebras-hecke}.  We retain the notation from these sections, in particular Notation~\ref{not:cyclic-reduction-map} and \ref{not:delocalised-trace}.

The next lemma gives an alternative perspective on the map $\Phi$ introduced in Notation~\ref{not:cyclic-reduction-map}, which cyclically reduces elements in Hecke algebras.  It will allow us to prove boundedness of $\Phi$ when extended to suitable Schwartz algebras.
\begin{lemma}
  \label{lem:cyclic-reduction-hecke-algebra}
  Let $(W, S)$ be a right-angled Coxeter group and $\cH$ a generic Hecke algebra of type $(W,S)$.  There are uniquely defined maps $\gamma, \rho \colon W \lra W$ such that for every $w \in W$ we have
  \begin{enumerate}
  \item $\rho(w)$ is a minimal length element in the conjugacy class of $w$;
  \item $\gamma(w) = e$ if $w$ is of minimal length in its conjugacy class and otherwise $\gamma(w) = s_1 \dotsm s_k$ where for all $1 \leq i \leq k$ the $\leq_S$-minimal element $s \in S$ satisfying $|s s_i \dotsm s_k \rho(w) s_k \dotsm s_i s| < |s_i \dotsm s_k \rho(w) s_k \dotsm s_i|$ is $s_i$;
  \item $w = \gamma(w) \rho(w) \gamma(w)^{-1}$ is a reduced decomposition.
  \end{enumerate}
  Furthermore, $\Phi\colon \cH \lra \cH$ is the unique $R$-linear map satisfying
  \begin{gather*}
    \Phi(T_w) = T_{\gamma(w)^{-1}} T_{\gamma(w)} T_{\rho(w)}
    \quad
    \text{ for all } w \in W
    \eqstop
  \end{gather*}
\end{lemma}
\begin{proof}
  Since every element in $\mathrm{Prod}(\gamma(w)^{-1}, \gamma(w)\rho(w))$ is of minimal length in its conjugacy class by Proposition~\ref{prop:minimal-representatives-hecke-conjugation}, the claim follows from the definition of $\Phi$.
\end{proof}

\begin{theorem}
  \label{thm:delocalised-trace-existence}
  Let $(W,S)$ be a right-angled, hyperbolic Coxeter system and let $\cH$  be the Iwahori-Hecke algebra of type $(W,S)$ with deformation parameter $(q_s)_{s \in S} \in \RR_{> 0}^S$.  Let $\cS'(W, S, q), \cS(W, S, q) \subseteq \Cstarred(W, S, q)$ be smooth subalgebras, such that $\cS'(W, S, q)$ is an unconditional completion of $\CC(W, S, q)$ and the natural quasi-derivation
  \begin{gather*}
    \Delta\colon T_w \mapsto \sum_{\substack{w = w_1w_2 \\ \text{ reduced decompositon}}} T_{w_1} \ot T_{w_2}
  \end{gather*}
  continuously extends to a map into the unconditional tensor product $\cS(W, S, q) \lra \cS'(W, S, q) \ot \cS'(W, S, q)$.  Then for every conjugacy class $\cO \subseteq W$ the trace $\vphi_\cO$ continuously extends to $\cS(W, S, q)$.
\end{theorem}
\begin{proof}
  Let us first observe that the map $\Sigma_\cO$ summing over the finitely many coefficients of elements of minimal length in $\cO$ extends to a continuous map on $\cS'(W, S, q)$, as $\cS'(W, S, q)$ is an unconditional completion of $\CC(W, S, q)$.  Consider maps $\gamma, \rho\colon W \lra W$ defined as in Lemma~\ref{lem:cyclic-reduction-hecke-algebra}.  Since $\Delta$ extends continuously to a map $\cS(W, S, q) \lra \cS'(W, S, q) \ot \cS'(W, S, q)$ and since $\cS'(W, S, q)$ is an unconditional completion of $\CC(W, S, q)$, the map $\Delta'\colon \cS(W, S, q) \lra \cS'(W, S, q) \ot \cS'(W, S, q)$ satisfying $\Delta'(T_w) = T_{\gamma(w) \rho(w)} \ot T_{\gamma(w)^{-1}}$ is well-defined and continuous.  Let $\Phi\colon \CC(W, S, q) \lra \CC(W, S, q)$ be the unique linear map satisfying $\Phi(T_w) = T_{\gamma(w)^{-1}} T_{\gamma(w)} T_{\rho(w)}$.  Then we have $\vphi_\cO = \Sigma_\cO \circ \Phi$ as shown in Theorem~\ref{thm:delocalised-trace-identification}.  So a continuous extension of $\vphi_\cO$ to $\cS(W, S, q)$ is obtained by the composition
  \begin{multline*}
    \cS(W, S, q) \stackrel{\Delta'}{\lra}
    \cS'(W, S, q) \ot \cS'(W, S, q) \stackrel{\text{tensor flip}}{\lra} \\
    \cS'(W, S, q) \ot \cS'(W, S, q) \stackrel{\text{multiplication}}{\lra}
    \cS'(W, S, q) \stackrel{\Sigma_\cO}{\lra}
    \CC
    \eqcomma
  \end{multline*}
This finishes the proof.
\end{proof}

We now give a proof of the simplification of Theorem~\ref{thm:delocalised-trace-existence} stated in the introduction.
\begin{proof}[Proof of Theorem~\ref{introthm:Schwartz-algebra-existence}]
  It suffices to note that the assumptions of Theorem~\ref{thm:delocalised-trace-existence} are satisfied thanks to Corollary~\ref{cor:hecke-schwartz-algebra-continuous-quasi-derivation}.
\end{proof}

\begin{definition}
  \label{def:delocalised-traces}
  Let $(W, S)$ be right-angled, hyperbolic Coxeter system and let $q \in \RR_{> 0}^S$ be a deformation parameter.  Given a conjugacy class $\cO \subseteq W$, we call a continuous extension of $\vphi_\cO$ to a smooth subalgebra of $\Cstarred(W, S, q)$ a delocalised trace associated to $\cO$.
\end{definition}

Before proving Theorem~\ref{thmintro:pairing}, let us clarify its context and notation beyond what we have done in the introduction already.  Recall from \cite[Theorem A.2.1]{bost1990-principe-oka} that the inclusion of a smooth subalgebra into a unital C*-algebra induces an isomorphism in K-theory.   We also recall the K-theory calculations for right-angled Hecke C*-algebras from \cite[Section 4]{raumskalski2022}.  To this end we use the description of conjugacy classes of finite order elements from Remark~\ref{rem:finite-order-element-cliques}.  Denote by $\mathrm{Cliq}(\Gamma)$ the set of cliques in the commutation graph $\Gamma$ of a right-angled Coxeter system $(W, S)$, that is subsets $C \subseteq S$ such that $st = ts$ for all $s,t \in C$.  For a clique $C \in \mathrm{Cliq}(\Gamma)$, we write $p_C = \prod_{s \in C} \chi_1^s$.  Here $\chi_1^s = \pi_s(\chi_1)$ denotes the image of the projection $\chi_1 \in \CC[\ZZ/2]$ associated with the trivial representation under the linear map extending the group homomorphism $\pi_s\colon \ZZ/2 \lra \CC(W, S, q)^\times$ which satisfies $\pi_s(1) = \frac{2}{q_s^{1/2} + q_s^{-1/2}}T_s + \frac{q_s^{-1/2} - q_s^{1/2}}{q_s^{1/2} + q_s^{-1/2}}$.  Then for every deformation parameter $q \in \RR_{> 0}^S$ the map $\mathrm{Cliq}(\Gamma) \lra \rK_0(\Cstarred(W, S, q))\colon C \mapsto [p_C]$ induces an isomorphism $\ZZ^{\mathrm{Cliq}(\Gamma)} \lra \rK_0(\Cstarred(W, S, q))$.
\begin{proof}[Proof of Theorem~\ref{thmintro:pairing}]
  As described above, for $D \in \mathrm{Cliq}(\Gamma)$ we have $p_D = \prod_{s \in D} \chi_1^s$ where $\chi_1^s$ is the projection associated with the trivial representation of the unitary representation $\ZZ/2 \mapsto \CC(W, S, q)$ mapping $1$ to $\frac{2}{q_s^{1/2} + q_s^{-1/2}}T_s + \frac{q_s^{-1/2} - q_s^{1/2}}{q_s^{1/2} + q_s^{-1/2}}$.  We have $\chi_1^s  = \mathbb{1}_{(0,\infty)}(T_s) = \frac{1}{q_s^{1/2} + q_s^{-1/2}} (T_s + q_s^{-1/2}) =\frac{1}{q_s^{1/2} + q_s^{-1/2}} T_s + \frac{1}{1 + q_s}$.  Writing $T_D = \prod_{s \in D} T_s$, we find that
  \begin{gather*}
    p_D = \sum_{C \subseteq D} \bigl ( \prod_{s \in C} \frac{1}{q_s^{1/2} + q_s^{-1/2}} \bigr ) \bigl ( \prod_{s \in D \setminus C} \frac{1}{1 + q_s} \bigr )  T_C
    \eqstop
  \end{gather*}
  Since all elements in the support of $p_D$ have minimal length and are pairwise non-conjugate, we find that
  \begin{gather*}
    \vphi_C(p_D)
    =
    \Sigma_C(p_D)
    =
    \begin{cases}
      0 & C \nsubseteq D \eqcomma \\
      \prod_{s \in C} \frac{1}{q_s^{1/2} + q_s^{-1/2}} \prod_{s \in D \setminus C} \frac{1}{1 + q_s} & C \subseteq D \eqstop
    \end{cases}
  \end{gather*}
  Now the claimed formula for $\vphi_C(p_D)$ follows.

  Let us now prove faithfulness of the pairing.  Let $a = \sum_{D \in \mathrm{Cliq}(\Gamma)} a_D [p_D]$ be an element in $\rK_0(\Cstarred(W, S, q))$ such that $\vphi_C(a) = 0$ for all $C \in \mathrm{Cliq}(\Gamma)$.  Assuming for a contradiction that $a \neq 0$, we may take $D \in \mathrm{Cliq}(\Gamma)$ maximal with the property that $a_D \neq 0$.  Then
  \begin{gather*}
    0 = \vphi_D(a) = \vphi_D(a_D p_D) = a_D \prod_{s \in D} \frac{1}{q_s^{1/2} + q_s^{-1/2}}
    \eqcomma
  \end{gather*}
  which is a contradiction.  Conversely, let $\vphi = \sum_{C \in \mathrm{Cliq}(\Gamma)} b_C \vphi_C$ be a finite $\RR$-linear combination of delocalised traces associated with finite order elements and assume that $\vphi(a) = 0$ for all $a \in \rK_0(\Cstarred(W, S, q))$.  Suppose then that $\vphi \neq 0$ and let $C \in \mathrm{Cliq}(\Gamma)$ be minimal with $b_C \neq 0$.  Then, in analogy to the calculation before, we find that
    \begin{gather*}
      0 = \vphi([p_C]) = b_C \vphi_C(p_C) = b_C \prod_{s \in C} \frac{1}{q_s^{1/2} + q_s^{-1/2}}
      \eqstop
  \end{gather*}
  This is a contradiction, finishing the proof of the theorem.
\end{proof}


\bibliography{delocalised-traces-right-angled-hyperbolic-hecke-algebras}
\bibliographystyle{plain}



\vspace{2em}

\noindent \begin{minipage}{1.0\linewidth}
  \small
  Piotr Nowak, Institute of Mathematics of the Polish Academy of Sciences, ul. Sniadeckich 8, 00-656 Warszawa, Poland \\
  E-mail address: pnowak@impan.pl
\end{minipage}

\vspace{1em}

\noindent \begin{minipage}[t]{\linewidth}
  \small
  Sanaz Pooya, Institute of Mathematics, University of Potsdam, Campus Golm, Haus 9, Karl-Liebknecht-Str. 24-25, Germany \\
  E-mail address: sanaz.pooya@uni-potsdam.de
\end{minipage}

\vspace{1em}

\noindent \begin{minipage}[t]{\linewidth}
  \small
  Sven Raum, Institute of Mathematics, University of Potsdam, Campus Golm, Haus 9, Karl-Liebknecht-Str. 24-25, Germany \\
  E-mail address: sven.raum@uni-potsdam.de
\end{minipage}

\vspace{1em}

\noindent \begin{minipage}{1.0\linewidth}
  \small
  Adam Skalski, Institute of Mathematics of the Polish Academy of Sciences, ul. Sniadeckich 8, 00-656 Warszawa, Poland \\
  E-mail address: a.skalski@impan.pl
\end{minipage}

\end{document}